\documentclass[11pt,reqno]{amsart}

\usepackage{amssymb}
\usepackage{amsmath}
\usepackage{amsthm}
\usepackage{amsfonts}
\usepackage{bbm}
\usepackage{mathrsfs}
\usepackage{centernot}
\usepackage{pgfplots}

\usepackage{a4wide}
\usepackage{bookmark}
\usepackage{hyperref}
\hypersetup{pdfstartview={FitH}}

\DeclareFontFamily{U}{mathx}{\hyphenchar\font45}
\DeclareFontShape{U}{mathx}{m}{n}{
<5> <6> <7> <8> <9> <10>
<10.95> <12> <14.4> <17.28> <20.74> <24.88>
mathx10
}{}
\DeclareSymbolFont{mathx}{U}{mathx}{m}{n}
\DeclareFontSubstitution{U}{mathx}{m}{n}
\DeclareMathAccent{\widecheck}{0}{mathx}{"71}
\DeclareMathAccent{\wideparen}{0}{mathx}{"75}

\newtheorem{theorem}{Theorem}
\newtheorem{lemma}[theorem]{Lemma}
\newtheorem{corollary}[theorem]{Corollary}
\newtheorem{proposition}[theorem]{Proposition}

\numberwithin{equation}{section}

\begin{document}

\title[Norm-variation of ergodic averages]{Norm-variation of ergodic averages with respect to two commuting transformations}

\author[P. Durcik]{Polona Durcik}
\address{Polona Durcik, Mathematisches Institut, Universit\"at Bonn, Endenicher Allee 60, 53115 Bonn, Germany}
\email{durcik@math.uni-bonn.de}

\author[V. Kova\v{c}]{Vjekoslav Kova\v{c}}
\address{Vjekoslav Kova\v{c}, Department of Mathematics, Faculty of Science, University of Zagreb, Bijeni\v{c}ka cesta 30, 10000 Zagreb, Croatia}
\email{vjekovac@math.hr}

\author[K. A. \v{S}kreb]{Kristina Ana \v{S}kreb}
\address{Kristina Ana \v{S}kreb, Faculty of Civil Engineering, University of Zagreb, Fra Andrije Ka\v{c}i\'{c}a Mio\v{s}i\'{c}a 26, 10000 Zagreb, Croatia}
\email{kskreb@grad.hr}

\author[C. Thiele]{Christoph Thiele}
\address{Christoph Thiele, Mathematisches Institut, Universit\"at Bonn, Endenicher Allee 60, 53115 Bonn, Germany}
\email{thiele@math.uni-bonn.de}

\date{\today}

\subjclass[2010]{Primary 37A30; Secondary 42B15, 42B20}


\begin{abstract}
We study double ergodic averages with respect to two general commuting transformations and establish a sharp quantitative result on their convergence in the norm. We approach the problem via real harmonic analysis, using recently developed methods for bounding multilinear singular integrals with certain entangled structure. A byproduct of our proof is a bound for a two-dimensional bilinear square function related to the so-called triangular Hilbert transform.
\end{abstract}

\maketitle


\section{Introduction}
Many problems in ergodic theory are related to the convergence of certain averages along the orbits with respect to one or several transformations. Let $(X,\mathcal{F},\mu)$ be a $\sigma$-finite measure space and let $S\colon X\to X$ be a measure-preserving transformation, i.e.\@ for any $E\in\mathcal{F}$ we have $S^{-1}E\in\mathcal{F}$ and $\mu(S^{-1}E)=\mu(E)$. The most classical result in this direction is von Neumann's \emph{mean ergodic theorem}~\cite{vn:erg}, which guarantees convergence of the \emph{single ergodic averages}
\begin{equation}\label{eq:singleaverages}
M_{n}f(x) := \frac{1}{n}\sum_{i=0}^{n-1} f(S^i x)
\end{equation}
in the $\textup{L}^2(X)$ norm for any $f\in\textup{L}^2(X)$. Classical proofs of this fact do not provide any information on the rate of this convergence. With the aid of the spectral theorem, Jones, Ostrovskii, and Rosenblatt~\cite{jor:L2} have observed the quantitative variant of this result in the form of the norm-variation estimate
\begin{equation}\label{eq:singlenormvar}
\sum_{j=1}^{m} \|M_{n_j}f - M_{n_{j-1}}f\|_{\textup{L}^2(X)}^2 \leq C \,\|f\|_{\textup{L}^2(X)}^2
\end{equation}
for any positive integers $n_0<n_1<\cdots<n_m$ and with an absolute finite constant $C$. The work of Bourgain \cite{jb:pt} prequels \eqref{eq:singlenormvar} and his pointwise variation estimates imply the same inequality albeit with the power $2$ replaced by an arbitrary $\varrho>2$. Calder\'{o}n's transference principle, a version of which we discuss in Section~\ref{sec:transition}, reduces \eqref{eq:singlenormvar} to studying operators in harmonic analysis that are well-understood by now.

Multiple ergodic averages were motivated by the work of Furstenberg and others \cite{hf:sz}, \cite{fk:msz}, \cite{fko:msz} connecting ergodic theory with arithmetic combinatorics. In this paper we are concerned with the bilinear case. Let $S,T\colon X\to X$ be two measure-$\mu$-preserving transformations such that $ST=TS$. For any two complex-valued measurable functions $f,g$ on $X$ and any positive integer $n$ one can define the \emph{double ergodic average} $M_{n}(f,g)$ as a function on $X$ given by
\begin{equation}\label{eq:averages}
M_{n}(f,g)(x) := \frac{1}{n}\sum_{i=0}^{n-1} f(S^i x) g(T^i x)
\end{equation}
for each $x\in X$. It is a classical result by Conze and Lesigne~\cite{cl:L2} that for any two functions $f,g\in\textup{L}^\infty(X)$ on a probability space the sequence of averages $(M_{n}(f,g))_{n=1}^{\infty}$ converges in the $\textup{L}^2$ norm. Standard density arguments combined with log-convexity of $\textup{L}^p$ norms extend this result to functions $f\in\textup{L}^{p_1}(X)$, $g\in\textup{L}^{p_2}(X)$, with convergence in the $\textup{L}^p$ norm, as long as the exponents satisfy $p<\infty$ and $1/p\geq 1/p_1+1/p_2$. However, no explicitly quantitative variant of this fact for completely general commuting transformations $S,T$ exists in the literature and this is the topic of the present paper.

Our main result is the following estimate for the averages \eqref{eq:averages}.

\begin{theorem}\label{thm:ergodicthm}
There is a finite constant $C$ such that for any $\sigma$-finite measure space $(X,\mathcal{F},\mu)$, any two commuting measure-preserving transformations $S,T$ on that space, and all functions $f,g\in\textup{L}^4(X)$ we have
\begin{equation}\label{eq:doublenormvar}
\sum_{j=1}^{m} \|M_{n_j}(f,g) - M_{n_{j-1}}(f,g)\|_{\textup{L}^2(X)}^2 \leq C \,\|f\|_{\textup{L}^4(X)}^2 \|g\|_{\textup{L}^4(X)}^2
\end{equation}
for each choice of positive integers $m$ and $n_0<n_1<\cdots<n_m$.
\end{theorem}

Such quantitative estimate for multiple ergodic averages was stated as an open problem by Avigad and Rute in the closing section of \cite{ar:osc}, after the question had already circulated in the community for a while. A result analogous to Theorem~\ref{thm:ergodicthm} was previously established by the second author in \cite{vk:dea}, but only for a simplified model, where the actions of $\mathbb{Z}$ are replaced by actions of infinite powers $\mathbb{A}^\omega$ of a fixed finite abelian group $\mathbb{A}$, and which avoided challenges we address in this paper.

Unlike for \eqref{eq:singlenormvar}, Calder\'on's transference of \eqref{eq:doublenormvar} leads to a non-classical problem in harmonic analysis, whose solution is the main point of our paper. We do not know of a martingale approach to \eqref{eq:doublenormvar}, even for particular cases of indices $n_j$. This is in contrast with the powerful martingale techniques for handling the single ergodic averages \eqref{eq:singleaverages}; compare with \cite{ar:osc}, \cite{jb:pt}, \cite{jsw:var}.

The techniques of this paper do not immediately generalize to the multiple variants of \eqref{eq:averages}, i.e.\@ to the analogous ergodic averages with respect to several commuting transformations. However, such averages are also known to converge in the norm, as was first shown by Tao~\cite{tt:L2}, with a different proof given by Austin~\cite{ta:L2}. More generally, norm convergence of multiple averages was established by Walsh~\cite{mw:nil} in the case when the transformations generate a nilpotent group.

Almost everywhere convergence of the averages \eqref{eq:averages} is a longstanding open problem. In the single average case \eqref{eq:singleaverages}, almost everywhere convergence is Birkhoff's classical \emph{pointwise ergodic theorem} \cite{gb:pt}, with quantitative estimates discussed in Bourgain \cite{jb:pt} and Jones, Kaufman, Rosenblatt, and Wierdl~\cite{jkrw:Lp}. For two transformations $S,T$ the task simplifies if $T$ is assumed to be a power of $S$, for instance $S$ is invertible and $T=S^{-1}$. It was successfully studied by the analytic approach and an almost everywhere convergence result was established by Bourgain~\cite{jb:conv}. Subsequently, a pointwise variation estimate was established by Do, Oberlin, and Palsson~\cite{dop:var}. The result from \cite{dop:var} also implies a variant of our Theorem~\ref{thm:ergodicthm} with exponent $\varrho>2$ in the special case $T=S^{-1}$. For further partial progress on a.e.\@ convergence for general commuting transformations we refer to the preprint by Donoso and Sun~\cite{ds:pt} and references therein. In \cite{ds:pt} the a.e.\@ convergence is verified under the additional assumption that $(X,\mathcal{F},\mu,S,T)$ forms a so-called distal system, i.e.\@ a certain iterated topological extension of the trivial system.

Recall that the number of \emph{$\varepsilon$-jumps} or \emph{$\varepsilon$-fluctuations} of a sequence $(a_n)_{n=1}^{\infty}$ in a Banach space $B$, in our case $\textup{L}^2(X)$, is defined as the supremum of the set of integers $J$ for which there exist indices
\[ m_1 < n_1 \leq m_2 < n_2 \leq \cdots \leq m_J < n_J \]
such that $\|a_{n_j}-a_{m_j}\|_B \geq \varepsilon$ for $j=1,2,\ldots,J$.
A direct consequence of our main theorem is that for all functions $f,g$ of norm one in $\textup{L}^4(X)$ the number of $\varepsilon$-jumps of the averages \eqref{eq:averages} is at most $C \varepsilon^{-2}$. In particular, the number of $\varepsilon$-jumps is finite for each $\varepsilon>0$, which implies norm convergence, i.e.\@ it reproves the result by Conze and Lesigne \cite{cl:L2}. It follows further that for any $\varepsilon>0$ the sequence $(M_{n}(f,g))_{n=1}^{\infty}$ can be covered by at most $C \varepsilon^{-2}+1$ balls of radius $\varepsilon$ in the Hilbert space $\textup{L}^2(X)$. Such a result is sometimes called a uniform bound for the metric entropy. It was shown by Bourgain~\cite{jb:entropy} that a.e.\@ convergence of certain sequences of functions, including the single ergodic averages \eqref{eq:singleaverages}, necessarily implies the uniform bound on their metric entropy. In that light Theorem~\ref{thm:ergodicthm} can also be thought of as a partial progress towards the conjecture on a.e.\@ convergence of \eqref{eq:averages}, even though the bilinear analogue of \cite{jb:entropy} does not appear in the literature.

Our main inequality may be reformulated as
\[ \|M_n(f,g)\|_{\textup{V}_n^\varrho(\mathbb{N},\textup{L}^{p}(X))} \leq C^{1/2} \,\|f\|_{\textup{L}^{p_1}(X)} \|g\|_{\textup{L}^{p_2}(X)} , \]
with $\varrho=p=2$ and $p_1=p_2=4$, where for $1\leq\varrho<\infty$ the \emph{$\varrho$-variation} of a Banach-space-valued function $a\colon\mathcal{U}\to B$ with $\mathcal{U}\subseteq\mathbb{R}$ is defined as
\[ \|a\|_{\textup{V}^\varrho(\mathcal{U},B)}:= \|a(t)\|_{\textup{V}_t^\varrho(\mathcal{U},B)}
:= \sup_{\substack{m\in\mathbb{N}\cup\{0\}\\ t_0,t_1,\ldots,t_m\in \mathcal{U}\\ t_0<t_1<\cdots<t_m}} \Big( \sum_{j=1}^{m} \|a(t_j)-a(t_{j-1})\|_{B}^\varrho \Big)^{1/\varrho}. \]
If $(X,\mathcal{F},\mu)$ is a probability space, then for any $f,g\in\textup{L}^\infty(X)$, $1\leq p<\infty$, and $\varrho\geq\max\{p,2\}$ we have
\[ \|M_n(f,g)\|_{\textup{V}_n^\varrho(\mathbb{N},\textup{L}^p(X))} \leq C_{p,\varrho} \,\|f\|_{\textup{L}^\infty(X)} \|g\|_{\textup{L}^\infty(X)} \]
for some finite constant $C_{p,\varrho}$ depending only on $p$ and $\varrho$. In order to see this, by the monotonicity of $\textup{L}^p$ norms on a probability space in the case $p<2$ we can use
\[ \|M_{n_j}(f,g)-M_{n_{j-1}}(f,g)\|_{\textup{L}^p(X)} \leq \|M_{n_j}(f,g)-M_{n_{j-1}}(f,g)\|_{\textup{L}^2(X)} \]
and by their log-convexity for $p>2$ we have
\begin{align*}
& \|M_{n_j}(f,g)-M_{n_{j-1}}(f,g)\|_{\textup{L}^p(X)} \\
& \leq (2\|f\|_{\textup{L}^\infty(X)}\|g\|_{\textup{L}^\infty(X)})^{1-2/p} \|M_{n_j}(f,g)-M_{n_{j-1}}(f,g)\|_{\textup{L}^2(X)}^{2/p}.
\end{align*}
We then apply \eqref{eq:doublenormvar} and for that purpose in the latter case we need $2\varrho/p\geq2$.

The variation exponent $2$ in Theorem~\ref{thm:ergodicthm} is the best possible one. To see this, it suffices to consider the special case $|f|=|g|$ and $S=T$ and notice that this special case is tantamount to estimate \eqref{eq:singlenormvar}, where the exponent $2$ is well known to be sharp. The range of exponents $p_1,p_2,p,\varrho$ in the above discussion is likely not exhausted as the analogous work \cite{vk:dea} in the simplified setting suggests.

This paper, while self-contained, builds on a technique for bounding multi-linear and multi-scale singular integral operators gradually developed by the authors in \cite{pd:L4}, \cite{pd:Lp}, \cite{vk:bell}, \cite{vk:tp}, \cite{vk:dea}, \cite{ks:mart}, \cite{kt:t1}. We consider the present application to quantitative norm convergence for double ergodic averages a milestone in these efforts. A notable difference from the almost everywhere result by Do, Oberlin, and Palsson~\cite{dop:var} is that we do not use wave packet analysis or time-frequency analysis, as these tools are not well-adapted to our problem.

The technique we use resembles energy methods in partial differential equations. The main ingredients are integration by parts, positivity arguments, and the Cauchy-Schwarz inequality. The idea is to set up a partial integration scheme to produce positive terms, similar to energies, and then use upper bounds on a sum of positive terms to control each term individually. Unlike for most energy arguments in partial differential equations, here the partial integration happens in the scale parameter, which is typical for the singular integral theory. The structural complexity of the problem requires to iterate these steps, with the Cauchy-Schwarz inequality used inbetween to reduce the complexity of the expressions.

Let us elaborate more on the harmonic analysis part of the paper.
For a one-dimensional integrable function $\varphi$ and two-dimensional functions $F,G\in\textup{L}^4(\mathbb{R}^2)$, for $t>0$, and for $(x,y)\in\mathbb{R}^2$ we introduce the bilinear averages
\[ A_t^{\varphi}(F,G)(x,y) := \int_{\mathbb{R}} F(x+s,y) G(x,y+s) \,t^{-1} \varphi(t^{-1}s) \,ds. \]
Theorem~\ref{thm:ergodicthm} will be a consequence of the following bilinear estimate where $\varphi=\mathbbm{1}_{[0,1)}$ is the characteristic function of the interval $[0,1)$.

\begin{theorem}\label{thm:analyticthmRough}
There exists a finite constant $C$ such that for any $F,G\in\textup{L}^4(\mathbb{R}^2)$ we have
\[ \big\|A^{\mathbbm{1}_{[0,1)}}_t (F,G)\big\|_{\textup{V}_t^2((0,\infty),\textup{L}^2(\mathbb{R}^2))} \leq C \,\|F\|_{\textup{L}^4(\mathbb{R}^2)} \|G\|_{\textup{L}^4(\mathbb{R}^2)}. \]
\end{theorem}

By invariance of the left hand side under rescaling in $t$ and by superposition, the theorem implies an inequality independent of the choice of positive numbers $t_0<\cdots<t_m$:
\begin{equation}\label{eq:expanded}
\sum_{j=1}^{m} \| A_{t_j}^\varphi (F,G)-A_{t_{j-1}}^\varphi (F,G) \|_{\textup{L}^2(\mathbb{R}^2)}^2
\leq C_\varphi^2 \,\|F\|_{\textup{L}^4(\mathbb{R}^2)}^2 \|G\|_{\textup{L}^4(\mathbb{R}^2)}^2,
\end{equation}
where
\[ \varphi(s) = \int_{(-\infty,0)} \mathbbm{1}_{[\alpha,0)}(s) \frac{d\mu(\alpha)}{-\alpha} + \int_{(0,\infty)} \mathbbm{1}_{[0,\alpha)}(s) \frac{d\mu(\alpha)}{\alpha} \]
for some finite complex Radon measure $\mu$ on $(-\infty,0)\cup (0,\infty)$. In particular, we get \eqref{eq:expanded} for compactly supported functions $\varphi$ of bounded variation and the constant $C_\varphi$ is then a universal multiple of the total mass of the measure $\mu$. Moreover, by choosing $d\mu(\alpha)=-\alpha\varphi'(\alpha)d\alpha$ we can recover an arbitrary Schwartz function $\varphi$ and in that case the constant $C_\varphi$ in \eqref{eq:expanded} is a multiple of $\int_{\mathbb{R}}|s\varphi'(s)|ds$.

In the proof of Theorem~\ref{thm:analyticthmRough} we gradually consider various classes of functions $\varphi$ and carefully control $C_\varphi$ for these classes. Indeed, we begin by showing that \eqref{eq:expanded} holds for an arbitrary Schwartz function. However, we will actually need to apply the theorem with $\varphi=\mathbbm{1}_{[0,1)}$, and this case is more subtle and requires more precise decay conditions in the auxiliary estimates. Prior to our paper, inequality \eqref{eq:expanded} was not known even for a single nonzero function $\varphi$.

Analytic reformulation of the aforementioned open problem on the a.e.\@ convergence of the averages \eqref{eq:averages} would require a strengthening of Theorem~\ref{thm:analyticthmRough} involving pointwise variation of the bilinear averages on $\mathbb{R}^2$. Even though our techniques are not sufficient for controlling the latter quantity in its full generality, we can still establish an estimate for the so-called ``short pointwise variation''. 
The following corollary is not really a consequence of Theorem~\ref{thm:analyticthmRough}, but rather a byproduct of Lemma~\ref{lemma:short} below and the discussion in Subsection~\ref{subsec:generalphi}. We formulate it here to emphasize that our short variation argument does not distinguish between pointwise and norm variations.

\begin{corollary}\label{cor:shortpointwise}
For any Schwartz function $\varphi$ there exists a finite constant $C_\varphi$ such that for any $F,G\in\textup{L}^4(\mathbb{R}^2)$ we have
\[ \Big( \sum_{i=-\infty}^\infty \big\| \|A_t^\varphi (F,G)(x,y)\|_{\textup{V}_t^2([2^i,2^{i+1}],\mathbb{C})} \big\|^2_{\textup{L}^2_{(x,y)}(\mathbb{R}^2)}\Big)^{1/2} \leq C_\varphi \|F\|_{\textup{L}^4(\mathbb{R}^2)} \|G\|_{\textup{L}^4(\mathbb{R}^2)}. \]
\end{corollary}

While deriving Theorem~\ref{thm:ergodicthm} from Theorem~\ref{thm:analyticthmRough}, the following discrete estimate will appear along the way. It is also worth stating as a separate corollary due to its elegant formulation. For any two double sequences $\widetilde{F},\widetilde{G}\colon\mathbb{Z}^2\to\mathbb{R}$, for $n\in\mathbb{N}$, and for $(k,l)\in\mathbb{Z}^2$ we define the discrete averages $\widetilde{A}_n$ by
\begin{equation}
\widetilde{A}_n(\widetilde{F},\widetilde{G})(k,l) := \frac{1}{n}\sum_{i=0}^{n-1} \widetilde{F}(k+i,l) \,\widetilde{G}(k,l+i) \label{eq:defineatilde}.
\end{equation}

\begin{corollary}\label{cor:newdiscest}
There exists a finite constant $C$ such that for any $\widetilde{F},\widetilde{G}\in\ell^4(\mathbb{Z}^2)$ we have
\[ \big\|\widetilde{A}_n (\widetilde{F},\widetilde{G})\big\|_{\textup{V}_n^2(\mathbb{N},\ell^2(\mathbb{Z}^2))} \leq C \,\|\widetilde{F}\|_{\ell^4(\mathbb{Z}^2)} \|\widetilde{G}\|_{\ell^4(\mathbb{Z}^2)}. \]
\end{corollary}

Inequality \eqref{eq:expanded}, even for Schwartz functions $\varphi$, is already new in the special case $t_j=2^j$.
In this case we set $\psi(s):=\varphi(s)-2\varphi(2s)$ and define the square function
\[ S(F,G)(x,y) := \bigg( \sum_{j\in\mathbb{Z}} \Big| \int_{\mathbb{R}} F(x+s,y) \,G(x,y+s) \,2^{-j} \psi(2^{-j}s) \,ds \Big|^2 \bigg)^{1/2}. \]
A simple limiting argument as $m\to \infty$ in \eqref{eq:expanded} yields the following corollary.

\begin{corollary}\label{cor:squarefn}
For any $F,G\in\textup{L}^4(\mathbb{R}^2)$ we have
\[ \|S(F,G)\|_{\textup{L}^2(\mathbb{R}^2)} \leq C_\psi \|F\|_{\textup{L}^4(\mathbb{R}^2)} \|G\|_{\textup{L}^4(\mathbb{R}^2)}, \]
with a finite constant $C_\psi$ depending on $\psi$ alone.
\end{corollary}

Indeed, square function estimates of this type are a stepping stone towards the proof of Theorem~\ref{thm:analyticthmRough}; for example compare with Proposition~\ref{prop:not-skip} stated in Section~\ref{sec:longshort}.

In contrast with Corollary~\ref{cor:squarefn}, no bounds are known for the corresponding bilinear singular integral
\[ T(F,G)(x,y) := \textup{p.v.}\int_{\mathbb{R}} F(x+s,y) \,G(x,y+s) \frac{ds}{s}, \]
which was introduced in \cite{dt:2dbht} and later named the \emph{triangular Hilbert transform}. Only partial results in this direction exist; see \cite{ktz:tht} for a particular case when one of the functions takes a special form. Moreover, Zorin-Kranich showed in \cite{pz:splx}, building on the approach of Tao~\cite{tt:mht}, that the truncations to $m$ consecutive scales,
\[ T_m(F,G)(x,y) := \sum_{j=1}^{m} \int_{\mathbb{R}} F(x+s,y) \,G(x,y+s) \,2^{-j} \psi(2^{-j}s) \,ds, \]
have norms from $\textup{L}^{p_1}(\mathbb{R}^2)\times\textup{L}^{p_2}(\mathbb{R}^2)$ to $\textup{L}^{p}(\mathbb{R}^2)$ that grow like $o(m)$ as $m\to\infty$, for any fixed choice of exponents $1<p,p_1,p_2<\infty$ such that $1/p=1/p_1+1/p_2$. Using Corollary~\ref{cor:squarefn} and the Cauchy-Schwarz inequality we   improve this growth to $O(m^{1/2})$ for $p=2$, $p_1=p_2=4$, and then the interpolation with the trivial estimates coming from H\"{o}lder's inequality gives the growth $O(m^{1-\epsilon})$ for general exponents $p,p_1,p_2$ as before and for some $\epsilon>0$ depending on them.

Furthermore, for given $f,g\in\textup{L}^4(\mathbb{R})$ let us take
\[ F(x,y) := f(x-y) R^{-1/4} \vartheta(R^{-1}y),\quad G(x,y) := g(x-y) R^{-1/4} \vartheta(R^{-1}x), \]
where $R>0$ and $\vartheta$ is a smooth compactly supported nonnegative function on $\mathbb{R}$ that is constantly $1$ on the interval $[-1,1]$. By substituting $z=x-y$, observing
\begin{align*}
& \int_{-R}^{R}\int_{-R}^{R} \Big| \int_{\mathbb{R}} F(x+s,y) \,G(x,y+s) \,2^{-j} \psi(2^{-j}s) \,ds \Big|^2 dxdy \\
& \geq \int_{-R}^{R} \Big|\int_{\mathbb{R}} f(z+s) \,g(z-s) \,2^{-j} \psi(2^{-j}s) \,ds\Big|^2 dz,
\end{align*}
applying Corollary~\ref{cor:squarefn}, and letting $R\to\infty$ we recover the $\textup{L}^4(\mathbb{R})\times\textup{L}^4(\mathbb{R})\to\textup{L}^2(\mathbb{R})$ estimate for the one-dimensional bilinear square function
\[ \widetilde{S}(f,g)(x) := \bigg( \sum_{j\in\mathbb{Z}} \Big|\int_{\mathbb{R}} f(x+s) \,g(x-s) \,2^{-j} \psi(2^{-j}s) \,ds\Big|^2 \bigg)^{1/2}. \]
The only previously known proof of an $\textup{L}^p$ bound for $\widetilde{S}$ employs wave-packet analysis, i.e.\@ it uses Khintchine's inequality to reduce to an average of a family of bilinear singular integrals parametrized by random signs and then recognizes these operators in the proof of boundedness of the bilinear Hilbert transform \cite{lt:bht1}, \cite{lt:bht2}.

Somewhat related, there is an open problem stated in the introductory section of the paper by Bernicot~\cite{fb:sq} to show $\textup{L}^p$ bounds for the bilinear square function
\[ S_{\Omega}(f,g)(x) := \bigg( \sum_{\omega\in\Omega} \Big|\int_{\mathbb{R}} f(x+s) \,g(x-s) \,\widecheck{\mathbbm{1}}_{\omega}(s) \,ds\Big|^2 \bigg)^{1/2} \]
for an arbitrary collection of disjoint intervals $\Omega$, which would be a bilinear variant of the well-known result by Rubio de Francia~\cite{rdf:sq}. Here $\widecheck{\mathbbm{1}}_{\omega}$ denotes the inverse Fourier transform of $\mathbbm{1}_{\omega}$. Bernicot~\cite{fb:sq} has verified this conjecture for a particular case of equidistant intervals of the same length, such as $\Omega=\{[j,j+1):j\in\mathbb{Z}\}$. The problem becomes simpler if we replace $\mathbbm{1}_{\omega}$ with a smooth bump function adapted to $\omega$, as was already observed by Lacey~\cite{ml:sq} in the case of the intervals $[j,j+1)$, see also \cite{bs:sq}, \cite{ms:sq}, \cite{rs:sq}. The above bilinear square function $\widetilde{S}$ is associated with smooth truncations of the lacunary intervals $\Omega=\{[2^j,2^{j+1}):j\in\mathbb{Z}\}$.

This paper is organized as follows: In Section~\ref{sec:longshort} we begin the proof of Theorem~\ref{thm:analyticthmRough} by splitting the jumps into the ``long ones'' (i.e.\@ those corresponding to the scales $t_j$ that are dyadic numbers $2^k$, $k\in\mathbb{Z}$) discussed in Lemma~\ref{lemma:long} and the ``short ones'' (i.e.\@ those corresponding to $t_j$ from a fixed interval $[2^k,2^{k+1}]$) discussed in Lemmata~\ref{lemma:short} and \ref{lemma:shortMeanZero}. Propositions~\ref{prop:skip} and \ref{prop:not-skip} are the key results here. Their proofs are postponed to Sections~\ref{sec:propskip} and \ref{sec:propnotskip} and these two sections contain the main novelties of our approach. Finally, the somewhat standard transition from Theorem~\ref{thm:analyticthmRough} to Corollary~\ref{cor:newdiscest} and then to Theorem~\ref{thm:ergodicthm} is presented in details in Section~\ref{sec:transition}.


\section{Averages on $\mathbbm{R}^2$, long and short variations}
\label{sec:longshort}
In this section we split Theorem~\ref{thm:analyticthmRough} into long and short variation estimates and show how to deduce these from Propositions~\ref{prop:skip}  and \ref{prop:not-skip} below.

For two non-negative quantities $A$ and $B$ we write $A\lesssim B$ if there exists a constant $C>0$ such that $A\leq C B$. When we want to emphasize dependence of the constant on some parameters $p,q,\ldots$, we denote them in the subscript, i.e.\@ we write $\lesssim_{p,q,\ldots}$. Occasionally we may omit writing down parameters that are understood. We write $A \sim B$ if both $A\lesssim B$ and $B\lesssim A$ are satisfied.

For a function $\varphi$ on $\mathbb{R}^d$ and $t>0$ we set $\varphi_t(x):=t^{-d}\varphi(t^{-1}x)$. Consequently, $A_{t}^{\varphi}=A_{1}^{\varphi_t}$. By $\mathcal{S}(\mathbb{R}^d)$ we denote the class of all Schwartz functions on $\mathbb{R}^d$, while the word ``smooth'' will always mean $\textup{C}^\infty$. The Fourier transform of an integrable function $\varphi$ on $\mathbb{R}^d$ is defined as
\[ \widehat{\varphi}(\xi) := \int_{\mathbb{R}^d} \varphi(x) e^{-2\pi i x\cdot\xi} dx, \]
so the Fourier inversion formula takes form
\[ \varphi(x) = \int_{\mathbb{R}^d} \widehat{\varphi}(\xi) e^{2\pi i x\cdot\xi} d\xi, \]
whenever $\varphi,\widehat{\varphi}\in\textup{L}^1(\mathbb{R}^d)$.
Derivatives of a single-variable function $\varphi$ will be denoted $\varphi'$, $\varphi''$, etc.\@ or $D\varphi$, $D^2\varphi$, etc., while we write $\partial^n\varphi$ for the partial derivatives. Let us remark that we reserve the notation $\varphi^{(n)}$ for the upper indices.

Now we can formulate the two propositions that will be the key ingredients in the proof of Theorem~\ref{thm:analyticthmRough}. Their own proofs will be postponed to the subsequent sections.

\begin{proposition} \label{prop:skip}
Let $\lambda>1$ and let $\vartheta, \varphi \in \mathcal{S}(\mathbb{R})$ be such that
\[ |\vartheta(s)| \leq (1+|s|)^{-\lambda},\quad |\varphi(s)| \leq (1+|s|)^{-\lambda} \]
for all $s\in \mathbb{R}$.
Moreover, assume that $\widehat{\vartheta}$ is supported in $[-2^{-4},2^{-4}]$, while $\widehat{\varphi}$ is supported in $[-1,1]$ and constant on $[-2^{-2},2^{-2}]$. Then for any $m\in\mathbb{N}$, $k_0,\dots,k_m\in \mathbb{Z}$, and for any real-valued $F,G\in\mathcal{S}(\mathbb{R}^2)$ normalized by \begin{equation}\label{eq:normalization}
\|F\|_{\textup{L}^4(\mathbb{R}^2)}= \|G\|_{\textup{L}^4(\mathbb{R}^2)}=1
\end{equation}
we have
\begin{align}
\bigg| \sum_{j=1}^m \int_{\mathbb{R}^4} F(x+u,y)G(x,y+u)F(x+v,y)G(x,y+v) & \nonumber \\[-1.5ex]
\vartheta_{2^{k_j}}(u)(\varphi_{2^{k_j}}-\varphi_{2^{k_{j-1}}})(v) \,dxdydudv & \bigg| \lesssim_\lambda 1. \label{eq:formprop5}
\end{align}
\end{proposition}

\begin{proposition} \label{prop:not-skip}
Let $\lambda>1$ and let $\Phi \in \mathcal{S}(\mathbb{R}^2)$ be such that
\begin{equation}\label{eq:cdprop6}
|\Phi(u,v)| \leq (1+|u+v|)^{-\lambda} (1+|u-v|)^{-2\lambda}
\end{equation}
for all $u,v\in \mathbb{R}$.
Moreover, assume that $\widehat{\Phi}$ is supported in $([-2,-2^{-5}] \cup [2^{-5},2])^2$. Then for any real-valued $F,G\in\mathcal{S}(\mathbb{R}^2)$ normalized as in \eqref{eq:normalization} and for any $N\in\mathbb{N}$ we have \begin{equation}\label{eq:estnontensor}
\bigg| \sum_{j=-N}^{N} \int_{\mathbb{R}^4} F(x+u,y)G(x,y+u)F(x+v,y)G(x,y+v) \Phi_{2^j}(u,v) \,dxdydudv \bigg| \lesssim_\lambda 1.
\end{equation}
\end{proposition}

Note that for $\nu=3\lambda$, the estimate
\begin{equation}\label{eq:cdprop6equiv}
|\Phi(u,v)| \leq (1+|u|)^{-\lambda/2}(1+|v|)^{-\lambda/2} (1+|u-v|)^{-\nu}
\end{equation}
implies \eqref{eq:cdprop6} within an absolute constant. Moreover, \eqref{eq:cdprop6} implies \eqref{eq:cdprop6equiv} with $\nu=\lambda$, modulo a constant. We will pass between the two formulations in the subsequent sections.

We also remark that the bump functions in \eqref{eq:formprop5} do not satisfy any estimates of the type   \eqref{eq:cdprop6} within an absolute constant since there is no control on $k_j-k_{j-1}$. However, the form in Proposition~\ref{prop:skip} has better cancellation properties than the one in Proposition~\ref{prop:not-skip}. The support of its multiplier symbol does not intersect the antidiagonal $\eta=-\xi$, which is the key property we need in the proof.

In the rest of this section we concentrate on deducing Theorem~\ref{thm:analyticthmRough} from these propositions. Throughout the text, $\chi$ will denote a fixed smooth frequency cutoff. More precisely, we fix a function $\chi$ such that its Fourier transform $\widehat{\chi}$ is smooth, even, non-negative, supported in $[-1,1]$, constantly equal to $1$ on $[-2^{-1},2^{-1}]$, and monotone on $[2^{-1},1]$. Moreover, we can achieve that $\widehat{\chi}$ is the square of some nonnegative smooth function. Any constants are allowed to depend on $\chi$ and this dependence will not be mentioned explicitly.

\subsection{Long variation}
The following lemma is derived from Propositions~\ref{prop:skip} and \ref{prop:not-skip}.

\begin{lemma}\label{lemma:long}
Let $\phi\in\mathcal{S}(\mathbb{R})$ and assume that for some $\lambda>1$ and constants $C_0,C_1$ one has
\begin{align}\label{eq:decaylong1}
|\phi\ast\chi_{2^4}(s)| \leq C_0 (1+|s|)^{-\lambda}, \quad
|\phi(s)| \leq C_1 (1+|s|)^{-\lambda}
\end{align}
for all $s\in\mathbb{R}$,
and that for some $\lambda>1$ and a constant $C_2$ one has
\begin{align}\label{eq:decaylong2}
|\phi(u) \overline{\phi(v)}| & \leq C_2 (1+|u+v|)^{-\lambda} (1+|u-v|)^{-2\lambda}
\end{align}
for all $u,v\in\mathbb{R}$.
Moreover, assume that $\widehat{\phi}$ is supported in $[-1,1]$ and constant on $[-2^{-2},2^{-2}]$.
If $F,G\in\textup{L}^4(\mathbb{R}^2)$ are normalized by \eqref{eq:normalization}, then
\begin{equation}\label{eq:long}
\|A_{2^k}^\phi (F,G)\|_{\textup{V}_k^2(\mathbb{Z},\textup{L}^2(\mathbb{R}^2))} \lesssim_\lambda C_0^{1/2}C_1^{1/2} + C_2^{1/2}.
\end{equation}
\end{lemma}

Observe that if $\widehat{\phi}$ vanishes on $[-2^{-2},2^{-2}]$, then the first estimate in \eqref{eq:decaylong1} holds with $C_0=0$. In this case Lemma \ref{lemma:long} yields
\begin{align}\label{lemma:longMeanZero}
\|A_{2^k}^{\phi} (F,G)\|_{\textup{V}_k^2(\mathbb{Z},\textup{L}^2(\mathbb{R}^2))} \lesssim_\lambda C_2^{1/2}.
\end{align}

\begin{proof}[Proof of Lemma~\ref{lemma:long}] Standard limiting arguments reduce the estimate~\eqref{eq:long} for each fixed choice of the integers $k_0<\cdots<k_m$ to the case of Schwartz functions $F$ and $G$. By splitting into real and imaginary parts and using Minkowski's inequality, we may assume that $F$, $G$, and $\phi$ take only real values.

Fix integers $k_0<k_1<\cdots<k_m$ and denote
\[ V(F,G) := \sum_{j=1}^m \big\|A^{\phi}_{2^{k_j}}(F,G)-A^{\phi}_{2^{k_{j-1}}}(F,G)\big\|_{\textup{L}^2(\mathbb{R}^2)}^2. \]
Expanding the $\textup{L}^2$ norm gives
\begin{align*}
V(F,G)= \sum_{j=1}^m \int_{\mathbb{R}^4} & F(x+u,y)G(x,y+u)F(x+v,y)G(x,y+v) \\[-1ex]
& (\phi_{2^{k_j}}- \phi_{2^{k_{j-1}}} )(u) (\phi_{2^{k_j}} - \phi_{2^{k_{j-1}}})(v) \,dx dy du dv.
\end{align*}
We have the identity
\begin{align}
(\phi_{2^{k_j}}- \phi_{2^{k_{j-1}}} )(u) (\phi_{2^{k_j}}- \phi_{2^{k_{j-1}}} )(v)
& = \big( \phi_{2^{k_{j-1}}}(u)\phi_{2^{k_{j-1}}}(v) - \phi_{2^{k_{j}}}(u)\phi_{2^{k_{j}}}(v)\big ) \nonumber \\
& \quad + \phi_{2^{k_{j}}}(u) (\phi_{2^{k_{j}}} - \phi_{2^{k_{j-1}}} )(v) \nonumber \\
& \quad + (\phi_{2^{k_{j}}}-\phi_{2^{k_{j-1}}} )(u)\phi_{2^{k_{j}}}(v) . \label{eq:telescope}
\end{align}
Summing \eqref{eq:telescope} over $1\leq j\leq m$, the first term on the right hand-side telescopes into
\[ \phi_{2^{k_{0}}}(u)\phi_{2^{k_{0}}}(v) - \phi_{2^{k_{m}}}(u)\phi_{2^{k_{m}}}(v). \]
Applying H\"older's inequality in $(x,y)$ for the exponents $(4,4,4,4)$ and using that \eqref{eq:decaylong2} implies $\int_{\mathbb{R}^2} |\phi(u)\phi(v)|dudv \lesssim_\lambda C_2$ we obtain
\begin{align}\nonumber
\Big| \int_{\mathbb{R}^4} & F(x+u,y)G(x,y+u)F(x+v,y)G(x,y+v) \\[-1ex]
& (\phi_{2^{k_{0}}}(u)\phi_{2^{k_{0}}}(v) - \phi_{2^{k_{m}}}(u)\phi_{2^{k_{m}}}(v)) \,dx dy du dv \Big|
\lesssim_\lambda C_2 \|F\|^2_{\textup{L}^4(\mathbb{R}^2)} \|G\|^2_{\textup{L}^4(\mathbb{R}^2)} = C_2. \label{eq:estsinglescale}
\end{align}
By symmetry of the second and the third term on the right hand side of \eqref{eq:telescope}, it then suffices to bound
\begin{align*}
\Lambda(F,G) := \sum_{j=1}^m \int_{\mathbb{R}^4} F(x+u,y)G(x,y+u)F(x+v,y)G(x,y+v) & \\[-1ex]
\phi_{2^{k_{j}}}(u)(\phi_{2^{k_{j}}}-\phi_{2^{k_{j-1}}})(v) \,dx dy du dv & .
\end{align*}

Now we localize the multiplier symbol associated with this form. Let $\omega$ be defined by $\omega := \chi_{2^{-1}}-\chi_{2^4}$. Note that $\widehat{\omega}$ is supported in $[-2,-2^{-5}]\cup [2^{-5},2]$ and that $\widehat{\chi_{2^4}}+\widehat{\omega}$ equals $1$ on $[-1,1]$, and in particular also on the support of $\widehat{\phi}$.
Then we can write
\[ \phi = \phi\ast\chi_{2^4} + \phi\ast\omega. \]
Using this decomposition we split $\Lambda = \Lambda_{\chi_{2^4}} + \Lambda_{\omega}$, where for a function $\rho$, the form $\Lambda_{\rho}$ is defined by
\begin{align*}
\Lambda_{\rho}(F,G) := \sum_{j=1}^m \int_{\mathbb{R}^4} F(x+u,y)G(x,y+u)F(x+v,y)G(x,y+v) & \\[-1ex]
(\phi\ast\rho)_{2^{k_j}}(u)(\phi_{2^{k_j}}-\phi_{2^{k_{j-1}}})(v) \,dx dy du dv. &
\end{align*}
By the assumptions \eqref{eq:decaylong1} on $\phi$, Proposition~\ref{prop:skip} gives
\begin{equation}\label{eq:lambdath}
|{\Lambda}_{\chi_{2^4}} (F,G)| \lesssim_{\lambda} C_0C_1.
\end{equation}
Rewrite $\Lambda_{\omega}$ by separating the functions in $u$ and $v$ as
\begin{align*}
\Lambda_{\omega}(F,G) = \sum_{j=1}^m \int_{\mathbb{R}^2} & \Big( \int_{\mathbb{R}}F(x+u,y)G(x,y+u)
(\phi\ast\omega)_{2^{k_{j}}}(u) du \Big) \\
& \Big( \int_{\mathbb{R}}F(x+v,y)G(x,y+v) (\phi_{2^{k_{j}}}-\phi_{2^{k_{j-1}}})(v) dv \Big) \,dx dy.
\end{align*}
Applying the Cauchy-Schwarz inequality in $x$, $y$, and $j$ gives
\begin{equation}\label{eq:auxlambdath}
|\Lambda_{\omega}(F,G)| \leq \widetilde{\Lambda}_{\omega}(F,G)^{1/2} V(F,G)^{1/2},
\end{equation}
where for a function $\rho$ we have set
\begin{align*}
\widetilde{\Lambda}_{\rho}(F,G) := \, \sum_{j=1}^m \int_{\mathbb{R}^2} \Big(\int_{\mathbb{R}} F(x+u,y) G(x,y+u) (\phi\ast\rho)_{2^{k_{j}}}(u)\,du\Big)^2 dxdy.
\end{align*}
Note that, up to increasing the quantity $\widetilde{\Lambda}_{\omega}(F,G)$ by adding nonnegative terms, we may assume that $k_j=j$ and that the summation is taken over all integers $j$ from a sufficiently large interval $[-N,N]$.
Expanding the square in $\widetilde{\Lambda}_{\omega}(F,G)$ we can write this form as
\[\sum_{j=1}^m \int_{\mathbb{R}^4} F(x+u,y)G(x,y+u)F(x+v,y)G(x,y+v)
(\phi\ast\omega)_{2^{k_{j}}}(u)(\phi\ast\omega)_{2^{k_{j}}}(v) \,dxdydudv. \]
By the assumption \eqref{eq:decaylong2}, Proposition~\ref{prop:not-skip} implies
\begin{equation}\label{eq:ttlambdath}
\widetilde{\Lambda}_{\omega}(F,G) \lesssim_{\lambda} C_2.
\end{equation}
Inequalities \eqref{eq:estsinglescale}, \eqref{eq:lambdath}, \eqref{eq:auxlambdath}, and \eqref{eq:ttlambdath} together give a bootstrapping estimate
\[ V(F,G) \lesssim_{\lambda} C_2 + C_0C_1+ C_2^{1/2}V(F,G)^{1/2}. \]
This shows $V(F,G) \lesssim_\lambda C_0C_1 + C_2$ and hence proves \eqref{eq:long}.
\end{proof}

\subsection{Short variation}
The following two closely related lemmata are derived from Proposition~\ref{prop:not-skip}.

\begin{lemma}\label{lemma:short}
Let $\phi\in \mathcal{S}(\mathbb{R})$ and assume that for some $\lambda>1$ and a constant $C_3$ one has
\begin{equation}\label{eq:decayshort}
\Big| \int_1^2 t\partial_t({\phi}_t(u)) t\partial_t(\overline{\phi_t(v)})\frac{dt}{t} \Big| \leq C_3 (1+|u+v|)^{-\lambda} (1+|u-v|)^{-2\lambda}
\end{equation}
for all $u,v\in\mathbb{R}$.
Moreover, assume that $\widehat{\phi}$ is supported in $[-1,1]$ and constant on $[-2^{-4},2^{-4}]$. If $F,G\in\textup{L}^4(\mathbb{R}^2)$ are normalized by \eqref{eq:normalization}, then for each $N\in \mathbb{N}$ one has
\begin{equation}\label{eq:short}
\Big( \sum_{i=-N}^N \big\| \|A_t^\phi (F,G)(x,y)\|_{\textup{V}_t^2([2^i,2^{i+1}],\mathbb{C})} \big\|^2_{\textup{L}^2_{(x,y)}(\mathbb{R}^2)}\Big)^{1/2} \lesssim_{\lambda} C_3^{1/2},
\end{equation}
with the implicit constant independent of $N$.
\end{lemma}

\begin{lemma}\label{lemma:shortMeanZero}
Let $\phi,F,G$ be as in the previous lemma. If in addition for some $\lambda>1$ and a constant $C_2$ the function $\phi$ satisfies \eqref{eq:decaylong2} for all $u,v\in \mathbb{R}$ and if $\widehat{\phi}$ vanishes on $[-2^{-4},2^{-4}]$, then for each $N\in\mathbb{N}$ we have the estimate
\begin{equation}\label{eq:short2}
\Big( \sum_{i=-N}^N \big\| \|A_t^\phi (F,G)(x,y)\|_{\textup{V}_t^2([2^i,2^{i+1}],\mathbb{C})} \big\|^2_{\textup{L}^2_{(x,y)}(\mathbb{R}^2)} \Big)^{1/2} \lesssim_{\lambda} C_2^{1/4} C_3^{1/4},
\end{equation}
with the implicit constant independent of $N$.
\end{lemma}

Observe that lemmata \ref{lemma:short} and \ref{lemma:shortMeanZero} actually establish pointwise short variation estimates. Since we clearly have
\[ \|A_t^\phi (F,G)\|_{\textup{V}_t^2([2^i,2^{i+1}],\textup{L}^2(\mathbb{R}^2))} 
\leq \big\| \|A_t^\phi (F,G)(x,y)\|_{\textup{V}_t^2([2^i,2^{i+1}],\mathbb{C})} \big\|_{\textup{L}^2_{(x,y)}(\mathbb{R}^2)}, \]
these in turn also imply the corresponding norm-variation estimates.

\begin{proof}[Proof of Lemma~\ref{lemma:short}]
As in the proof of Lemma \ref{lemma:long} we may assume that $F,\,G\in \mathcal{S}(\mathbb{R}^2)$ and that $F,\,G,$ and $\phi$ are real-valued.

Denote $\psi(s):= (s\phi(s))'$, so that one has $\psi_t(s)=-t\partial_t (\phi_t(s))$.
By Lemma \ref{lemma:peetre}   (in the Appendix) applied with $a(t)= A^\phi_t(F,G)(x,y)$ for each fixed $(x,y)$
we have
\[ \sup_{2^{i}\leq t_0<\cdots< t_m\leq 2^{i+1}}\sum_{j=1}^m |A^{\phi}_{t_j}(F,G)(x,y)-A^{\phi}_{t_{j-1}}(F,G)(x,y)|^2 \leq   \int_{1}^2 \big( A_{2^it}^{\psi}(F,G)(x,y) \big)^2 \frac{dt}{t}  . \]
Indeed, this follows from $A_t^\psi(F,G) = -t\partial_t(A^{\phi}_t(F,G))$ and by rescaling in $t$.
Integrating in $x,y$ and  summing over $-N\leq i\leq N$ yields
\[ \sum_{i=-N}^N \big\| \|A_t^\phi (F,G)(x,y)\|_{\textup{V}_t^2([2^i,2^{i+1}],\mathbb{C})} \big\|^2_{\textup{L}^2_{(x,y)}(\mathbb{R}^2)} \leq \sum_{i=-N}^N\int_{\mathbb{R}^2} \int_{1}^2 \big( A_{2^it}^{\psi}(F,G)(x,y) \big)^2 \frac{dt}{t} dx dy. \]
Expanding the square on the right hand-side, in order to finish the proof of Lemma \ref{lemma:short} we need to bound
\begin{equation}\label{eq:form2tobound}
\sum_{i=-N}^N \int_{\mathbb{R}^4} F(x+u,y)G(x,y+u)F(x+v,y)G(x,y+v) \,\Big( \int_{1}^{2} \psi_{2^i t} (u) \psi_{2^i t} (v) \frac{dt}{t} \Big) \,dxdydudv.
\end{equation}
Observe that $\widehat{\psi}(\xi) = -\xi \widehat{\phi}'(\xi)$ is supported in $[-1,-2^{-4}]\cup[2^{-4},1]$, so
\[ \Phi(u,v):=\int_{1}^{2}\psi_{t} (u) \psi_{t} (v) \frac{dt}{t} \]
has its frequency support in $([-1,-2^{-5}]\cup[2^{-5},1])^2$, and recall that we assume \eqref{eq:decayshort}. Proposition~\ref{prop:not-skip} implies boundedness of \eqref{eq:form2tobound} within an absolute constant times $C_3$, which yields \eqref{eq:short}.
\end{proof}

\begin{proof}[Proof of Lemma \ref{lemma:shortMeanZero}]
Let all the notation and the assumptions be as in the proof of the previous lemma.
By Lemma \ref{lemma:peetre} and the Cauchy-Schwarz inequality in $x,y$ and   $i$  we deduce
\begin{align*} &\sum_{i=-N}^N \big\| \|A_t^\phi (F,G)(x,y)\|_{\textup{V}_t^2([2^i,2^{i+1}],\mathbb{C})} \big\|^2_{\textup{L}^2_{(x,y)}(\mathbb{R}^2)}\\
& \lesssim \prod_{\rho\in\{\phi,\psi\}}\Big( \sum_{i=-N}^N \int_{\mathbb{R}^2} \int_{1}^2 \big( A_{2^it}^{\rho}(F,G)(x,y) \big)^2 \frac{dt}{t} dx dy \Big)^{1/2}.
\end{align*}
By the support assumptions on $\phi$, \eqref{eq:decaylong2}, and \eqref{eq:decayshort}, Proposition \ref{prop:not-skip} applied twice gives that the right hand-side is no greater than an absolute constant times $C_2^{1/2}C_3^{1/2}$, which in turn implies \eqref{eq:short2}.
\end{proof}

Finally, we are ready to deduce Theorem \ref{thm:analyticthmRough} from these lemmata. The first step is to show the estimate \eqref{eq:expanded} for a general Schwartz function $\varphi$.

\subsection{Deriving Theorem \ref{thm:analyticthmRough} for a Schwartz function $\varphi$}
\label{subsec:generalphi}
Let $F,G\in\mathcal{S}(\mathbb{R}^2)$ be normalized by \eqref{eq:normalization}. If $\varphi\in\mathcal{S}(\mathbb{R})$ is
such that $\widehat{\varphi}$ is supported in $[-1,1]$ and constant on $[-2^{-2},2^{-2}]$, then Lemmata~\ref{lemma:long} and \ref{lemma:short} combined with the standard separation into long and short jumps imply
\begin{equation}\label{eq:rhoauxvarintro}
\big\|A_t^{\varphi}(F,G)\big\|_{\textup{V}_t^2((0,\infty),\textup{L}^2(\mathbb{R}^2))} \lesssim_{\lambda} C_0^{1/2}C_1^{1/2} + C_2^{1/2} + C_3^{1/2} \lesssim_{\varphi} 1.
\end{equation}
The details can be found for instance in \cite{dop:var} or \cite{jsw:var}. Note that the constants $C_i$ depend only on some Schwartz norm of $\varphi$ of a sufficiently large degree. This gives \eqref{eq:expanded} in the particular case.

Now we show \eqref{eq:rhoauxvarintro} for a general Schwartz function $\varphi$. Take $\varphi\in \mathcal{S}(\mathbb{R})$ and denote $\theta:=\chi-\chi_{2}$.
Observe that $\widehat{\theta}$ is supported in $[-1,-2^{-2}]\cup [2^{-2},1]$ and that
\begin{equation}\label{eq:partunity}
\sum_{k\in \mathbb{Z}} \widehat{\theta}(2^k\xi) =1
\end{equation}
for all $0\neq\xi\in\mathbb{R}$. Then we can write
\begin{equation}\label{eq:thm3decomp}
\varphi = c\chi + (\varphi - c\chi) = c\chi + \sum_{k\in \mathbb{Z}} (\varphi-c\chi)\ast\theta_{2^k},
\end{equation}
where the number $c$ is chosen such that $\widehat{\varphi}(0)-c\widehat{\chi}(0)=0$, i.e.\@ $c=\widehat{\varphi}(0)$. Note that the series in \eqref{eq:thm3decomp} converges pointwise (in any summation order) since $\varphi-c\chi$ and $\theta$ are Schwartz and $\theta$ has mean zero.

We proceed by bounding norm-variation of bilinear averages corresponding to the individual terms in the expansion \eqref{eq:thm3decomp}. For the part associated with $c\chi$ boundedness follows from \eqref{eq:rhoauxvarintro} since
$\chi$ is Schwartz and $\widehat{\chi}$ is constant near the origin:
\begin{equation}\label{eq:auxvarinthm3a}
\|A_t^{c\chi}(F,G)\|_{\textup{V}_t^2((0,\infty),\textup{L}^2(\mathbb{R}^2))} \lesssim 1.
\end{equation}
For the part associated with $(\varphi-c\chi)\ast\theta_{2^k}$ we show that the function $\vartheta = \vartheta^{(k)}$ defined by
\[ \vartheta := (\varphi-c\chi)_{2^{-k}} \ast\theta \]
satisfies the estimate
\begin{equation}\label{eq:auxvarinthm3b}
\big\|A_t^{\vartheta}\!(F,G)\big\|_{\textup{V}_t^2((0,\infty),\textup{L}^2(\mathbb{R}^2))} \lesssim 2^{-|k|}
\end{equation}
for any $k\in \mathbb{Z}$.
By scaling invariance of the left hand-side of \eqref{eq:auxvarinthm3b} the same estimate remains to hold for $\vartheta_{2^k} = (\varphi-c\chi)\ast\theta_{2^k}$, i.e.\@ for each term in the series expansion \eqref{eq:thm3decomp}.
Then, from \eqref{eq:thm3decomp}, \eqref{eq:auxvarinthm3a}, \eqref{eq:auxvarinthm3b}, Minkowski's inequality, and Fatou's lemma we obtain
\[ \|A_t^\varphi(F,G)\|_{\textup{V}_t^2((0,\infty),\textup{L}^2(\mathbb{R}^2))} \lesssim 1 + \sum_{k\in \mathbb{Z}}2^{-|k|} \lesssim 1, \]
which finishes the proof.

In order to verify \eqref{eq:auxvarinthm3b}, observe that $\widehat{\vartheta}$ is supported in $[-1,-2^{-2}]\cup [2^{-2},1]$, so in particular it is constant on $[-2^{-2}, 2^{-2}]$.
Since $\widehat{\varphi}-c\widehat{\chi}$ vanishes at zero, we have $|\widehat{\varphi}(\xi)-c\widehat{\chi}(\xi)| \lesssim_{\varphi} \min\{|\xi|,|\xi|^{-1}\}$ and hence, by $\widehat{\vartheta}(\xi)=(\widehat{\varphi}-c\widehat{\chi})(2^{-k}\xi)\widehat{\theta}(\xi)$ and the product rule,
\begin{align*}
\big\| |\xi|^\alpha D^\beta \widehat{\vartheta}(\xi) \big\|_{\textup{L}_\xi^\infty(\mathbb{R})} \lesssim_{\alpha,\beta} 2^{-|k|}
\end{align*}
for any $\alpha,\beta\geq 0$. Therefore, $2^{|k|}\vartheta$ satisfies \eqref{eq:decaylong1}, \eqref{eq:decaylong2}, and \eqref{eq:decayshort} with the constants independent of $k$. The estimate \eqref{eq:auxvarinthm3b} then follows from \eqref{eq:rhoauxvarintro} applied with $\varphi=2^{|k|}\vartheta$ and by   homogeneity.

Let us remark that the same arguments also establish Corollary~\ref{cor:shortpointwise}. We simply use Lemma~\ref{lemma:short}, this time to get a short pointwise variation estimate in the same particular case od $\varphi$, and then perform decomposition \eqref{eq:thm3decomp} of a general Schwartz function.

\subsection{Deriving Theorem \ref{thm:analyticthmRough} for $\varphi=\mathbbm{1}_{[0,1)}$}
Once again we can work with Schwartz functions $F$ and $G$ only. Let $F,G\in\mathcal{S}(\mathbb{R}^2)$ be normalized by \eqref{eq:normalization} and let $\chi,\theta$ be as in the previous subsection.
We have
\begin{align}\label{eq:seriesexpansion}
\mathbbm{1}_{[0,1)} = \mathbbm{1}_{[0,1)}\ast\chi + \sum_{k=-\infty}^{-1} \mathbbm{1}_{[0,1)}\ast\theta_{2^k} .
\end{align}
By the Plancherel identity the series in \eqref{eq:seriesexpansion} converges in the $\textup{L}^2$ norm. However, the same series also converges a.e., which follows from the weak $\textup{L}^2$ boundedness of the maximally truncated convolution-type singular integrals. Alternatively, we can pass to an a.e.\@ convergent subsequence of partial sums, as taking the limit over a subsequence is enough for our intended application.

By the discussion in Subsection \ref{subsec:generalphi} we obtain
\begin{equation}\label{eq:sigmaest}
\big\|A_t^{\mathbbm{1}_{[0,1)}\ast\chi}(F,G)\big\|_{\textup{V}_t^2((0,\infty),\textup{L}^2(\mathbb{R}^2))} \lesssim 1.
\end{equation}
Now we concentrate on the individual terms in \eqref{eq:seriesexpansion} for negative values of $k$. By $\widetilde{\theta}$ we denote the primitive of $\theta$, i.e.\@ $\widetilde{\theta}(s):=\int_{-\infty}^{s}\theta(u)du$. Observe that, since $\theta$ has integral zero, its primitive $\widetilde{\theta}$ decays rapidly.
The arguments from the previous subsection give
\begin{equation} \label{lemma:phi1}
\|A_t^{\widetilde{\theta}} (F,G)\|_{\textup{V}_t^2((0,\infty),\textup{L}^2(\mathbb{R}^2))} \lesssim 1.
\end{equation}
By scaling invariance of the left hand-side, \eqref{lemma:phi1} also holds with $\widetilde{\theta}$ replaced by $\widetilde{\theta}_{2^k}$.
We will show that for each $k<0$ and for the function $\vartheta= \vartheta^{(k)}$ defined by
\[ \vartheta(s):= 2^k\widetilde{\theta}(s-2^{-k}) \]
we have the variational inequality
\begin{equation} \label{eq:lemmaphi}
\|A_t^{\vartheta} (F,G)\|_{\textup{V}_t^2((0,\infty),\textup{L}^2(\mathbb{R}^2))} \lesssim 2^{k/8}.
\end{equation}
Once this is shown, by scaling invariance of the left hand-side, the estimate \eqref{eq:lemmaphi} remains to hold with $\vartheta$ replaced by $\vartheta_{2^k}$.
Then we need to observe that
\[ \mathbbm{1}_{[0,1)}\ast\theta_{2^k} = 2^k\widetilde{\theta}_{2^k} - \vartheta_{2^k}. \]
From \eqref{eq:seriesexpansion}, \eqref{eq:sigmaest}, \eqref{lemma:phi1}, \eqref{eq:lemmaphi}, Minkowski's inequality, and Fatou's lemma we finally obtain
\[ \|A_t^{\mathbbm{1}_{[0,1)}}(F,G)\|_{\textup{V}_t^2((0,\infty),\textup{L}^2(\mathbb{R}^2))} \lesssim 1 + \sum_{k\leq-1} (2^k + 2^{k/8}) \lesssim 1. \]

In order to see \eqref{eq:lemmaphi}, note that the Fourier support of ${\vartheta}$ is contained in $[-1,-2^{-2}]\cup [2^{-2},1]$.
For any $\lambda>0$, $\nu>0$, and $k<0$ we claim that
\begin{align} \label{eq:itmdecay1}
|\vartheta(u)\vartheta(v)| & \lesssim_{\lambda,\nu} 2^{k(2-\lambda)} (1+|u|)^{-\lambda/2} (1+|v|)^{-\lambda/2} (1+|u-v|)^{-\nu},\\
\label{eq:itmdecay2}
\Big|\int_1^2 t\partial_t(\vartheta_t(u))t\partial_t(\vartheta_t(v)) \frac{dt}{t} \Big| & \lesssim_{\lambda,\nu} 2^{k(1-\lambda)}(1+|u|)^{-\lambda/2} (1+|v|)^{-\lambda/2} (1+|u-v|)^{-\nu}.
\end{align}
We have already commented how bounds of this form with $\nu=3\lambda$ transform into bounds \eqref{eq:decaylong2} and \eqref{eq:decayshort}.
Once these two estimates are verified, the separation into short and long jumps together with \eqref{lemma:longMeanZero} and Lemma \ref{lemma:shortMeanZero}, which require \eqref{eq:itmdecay1} and
\eqref{eq:itmdecay2} to hold with $\lambda>1$, give
\begin{equation*}
\big\|A_t^{\vartheta}(F,G)\big\|_{\textup{V}_t^2((0,\infty),\textup{L}^2(\mathbb{R}^2))} \lesssim_{\lambda} C_2^{1/2} + C_2^{1/4}C_3^{1/4}
\end{equation*}
with $C_2\sim 2^{k(2-\lambda)}$ and $C_3\sim 2^{k(1-\lambda)}$. Choosing $\lambda=5/4$ we obtain \eqref{eq:lemmaphi}.

\begin{proof}[Proof of \eqref{eq:itmdecay1}] By the rapid decay of $\widetilde{\theta}$ we have
\begin{align*}
|\tilde{\theta}(u-2^{-k}) \tilde{\theta}(v-2^{-k})| & \lesssim_{\lambda,\nu} (1+|u-2^{-k}|)^{-\lambda/2-\nu}(1+|v-2^{-k}|)^{-\lambda/2-\nu} \\
& \leq (1+|u-2^{-k}|)^{-\lambda/2}(1+|v-2^{-k}|)^{-\lambda/2}(1+|u-v|)^{-\nu},
\end{align*}
where we used $|u-v| \leq |u-2^{-k}| + |v-2^{-k}|$. From
\begin{align*}
(1+|u-2^{-k}|)^{-\lambda/2} \leq (1+|u|)^{-\lambda/2} (1+2^{-k})^{\lambda/2} \lesssim_{\lambda} (1+|u|)^{-\lambda/2} 2^{-k\lambda/2}
\end{align*}
we then conclude \eqref{eq:itmdecay1}.
\end{proof}

\begin{proof}[Proof of \eqref{eq:itmdecay2}] Observe that $-t\partial_t(\vartheta_t(s)) = \vartheta_t(s) + (s\vartheta'(s))_t$.
Thus, $t\partial_t(\vartheta_t(u))t\partial_t(\vartheta_t(v))$ consist of four terms.
We will show \eqref{eq:itmdecay2} corresponding to $(s\vartheta'(s))_t $, that is,
\begin{align}\label{eq:decay3}
& \Big|\int_1^2 (u\vartheta'(u))_t(v\vartheta'(v))_t \frac{dt}{t} \Big| \lesssim_{\lambda,\nu} 2^{k(1-\lambda)}(1+|u|)^{-\lambda/2} (1+|v|)^{-\lambda/2} (1+|u-v|)^{-\nu}.
\end{align}
The analogous inequalities corresponding to the other terms are treated in the same manner.
To see \eqref{eq:decay3} we first observe
\[ (s\vartheta'(s))_t = (s2^k\theta(s-2^{-k}))_t = st^{-1}2^k \theta_t(s-t2^{-k}) \]
and bound $|\theta_t(s)|\lesssim_{\lambda,\nu} (1+|s|)^{-\lambda/2-\nu-1}$ using $t\in [1,2]$. Then we estimate
\begin{align*}
& \Big|\int_1^2 u\theta_t (u-t2^{-k}) v\theta_t (v-t2^{-k})\frac{dt}{t^3} \Big| \\
& \lesssim_{\lambda,\nu} |uv| \int_1^2 \big( (1+|u-t2^{-k}|) (1+|v-t2^{-k}|) \big)^{-\lambda/2-\nu-1}{dt}.
\end{align*}
By the triangle inequality $|u-v| \leq |u-t2^{-k}|+|v-t2^{-k}|$ and the Cauchy-Schwarz inequality in $t$, this is bounded by
\[ (1+|u-v|)^{-\nu}
|u|\Big(\int_1^2 (1+|u-t2^{-k}|)^{-\lambda-2} dt \Big)^{1/2}
|v|\Big(\int_1^2 (1+|v-t2^{-k}|)^{-\lambda-2} dt \Big)^{1/2}. \]
Now, if $|u|\le 2^{-k+2}$, then we estimate
\begin{align*}
& (1+|u|)^{\lambda/2} |u| \Big(\int_1^2 (1+|u-t2^{-k}|)^{-\lambda-2} dt \Big)^{1/2} \\
& \leq (1+|u|)^{\lambda/2+1} \Big(\int_{-\infty}^\infty (1+|u-t2^{-k}|)^{-\lambda-2} dt \Big)^{1/2} \\
& \lesssim_\lambda 2^{k/2}(1+|u|)^{\lambda/2+1}\lesssim_\lambda 2^{k(-1-\lambda)/2},
\end{align*}
where the second inequality follows by integrating in $t$.
If $|u|\ge 2^{-k+2}$, then we have $|u-t2^{-k}|\geq |u|/2$ and hence
\begin{align*}
& (1+|u|)^{\lambda/2} |u| \Big(\int_1^2 (1+|u-t2^{-k}| )^{-\lambda-2} dt \Big)^{1/2} \\
& \lesssim_\lambda (1+|u|)^{\lambda/2+1}(1+|u|)^{-\lambda/2-1} = 1 \leq 2^{k(-1-\lambda)/2}.
\end{align*}
The same estimates hold for the terms with $v$. After multiplication by $2^{2k}$ and division by $(1+|u|)^{\lambda/2}(1+|v|)^{\lambda/2}$ this shows \eqref{eq:decay3}.
\end{proof}


\section{Proof of Proposition~\ref{prop:skip}}
\label{sec:propskip}
Let us rewrite the form \eqref{eq:formprop5} from Proposition~\ref{prop:skip} in a more convenient way. Denote $\psi:=\varphi-\varphi_2$. Then we have the telescoping identity
\begin{equation}\label{eq:ftcexpand}
\varphi_{2^{k_{j-1}}}-\varphi_{2^{k_{j}}} =\sum_{l=k_{j-1}}^{k_j-1} \psi_{2^l} .
\end{equation}
We insert \eqref{eq:ftcexpand} into \eqref{eq:formprop5} and substitute
\[ x'=x+y+u,\quad y'=x+y+v,\quad \widetilde{F}(y,x'):=F(x'-y,y),\quad \widetilde{G}(x,x'):=G(x,x'-x). \]
Note that we still have $\|\widetilde{F}\|_{\textup{L}^4(\mathbb{R}^2)}= \|\widetilde{G}\|_{\textup{L}^4(\mathbb{R}^2)}=1$. Omitting the tildas for notational simplicity, it then suffices to show the inequality
\begin{align*}
\bigg| \sum_{j=1}^m \sum_{l=k_{j-1}}^{k_j-1} \int_{\mathbb{R}^4} & F(y,x') G(x,x')F(y,y')G(x,y') \\[-1ex]
& {\vartheta}_{2^{k_j}}(x'-x-y)\psi_{2^l}(y'-x-y) \,dx dy dx' dy' \bigg| \lesssim 1.
\end{align*}

First, we would like to write the kernel as a superposition of elementary tensors in the four variables $x,y,x',y'$. Using the Fourier inversion formula we write
\[ \vartheta_{2^{k_j}}(x'-x-y)\psi_{2^l}(y'-x-y) = \int_{\mathbb{R}^2} \widehat{{\vartheta}}(2^{k_j}\xi)\widehat{\psi}(2^l\eta)e^{2\pi i \xi (x'-x-y)}e^{2\pi i \eta (y'-x-y)} d\xi d\eta. \]
Since $\widehat{\varphi}$ is supported in $[-1,1]$ and constant on $[-2^{-2},2^{-2}]$, the function $\widehat{\psi}$ is supported in $[-1,-2^{-3}]\cup [2^{-3},1]$. If $2^{k_j}\xi\in \mathrm{supp}(\widehat{{\vartheta}})$ and $2^l\eta\in \mathrm{supp}(\widehat{{\psi}})$, then
\[ 2^l(\xi + \eta) = 2^{l-k_j} 2^{k_j}\xi + 2^l\eta \in [-2,-2^{-4}]\cup[2^{-4},2]. \]
Let $\chi$ be as before, which guarantees that there exists a smooth nonnegative even function $\widehat{\omega}$, being the Fourier transform of some $\omega\in\mathcal{S}(\mathbb{R})$, satisfying
\[ \widehat{\omega}(\xi)^2 = \widehat{\chi}(2^{-2}\xi) - \widehat{\chi}(2^4 \xi). \]
The function $\widehat{\omega}$ is supported in $[-2^2,-2^{-5}]\cup[2^{-5},2^2]$ and equal to $1$ on $[-2,-2^{-4}]\cup[2^{-4},2]$. For each $(\xi,\eta)\in \mathbb{R}^2$ we have
\begin{equation}\label{eq:addomega}
\widehat{\vartheta}(2^{k_j}\xi)\widehat{\psi}(2^l\eta) = \widehat{\vartheta}(2^{k_j}\xi) \widehat{\psi}(2^l\eta) \widehat{\omega}(2^l(\xi+\eta))^2
\end{equation}
and hence
\begin{align*}
& \vartheta_{2^{k_j}}(x'-x-y)\psi_{2^l}(y'-x-y) \\
& = \int_{\mathbb{R}^2} \widehat{{\vartheta}}(2^{k_j}\xi) e^{2\pi i x'\xi} \widehat{\psi}(2^l\eta) e^{2\pi i y'\eta}
\widehat{\omega}(2^l(-\xi-\eta))
e^{2\pi i x(-\xi-\eta)} \widehat{\omega}(2^l(-\xi-\eta)) e^{2\pi i y(-\xi-\eta)} d\xi d\eta.
\end{align*}
The last expression can be viewed as the integral of the Fourier transform of the function
\[ \mathcal{H}(x_1,x_2,x_3,x_4):=\vartheta_{2^{k_j}}(x_1+x')\psi_{2^l}(x_2+y') \omega_{2^l}(x_3+x) \omega_{2^l}(x_4+y) \]
over the hyperplane
\begin{equation*}
\{(\xi,\eta,-\xi-\eta,-\xi-\eta) : \xi,\eta\in \mathbb{R}\}.
\end{equation*}
It equals the integral of $\mathcal{H}$ itself over the perpendicular hyperplane
\begin{equation*}
\{(p+q,p+q,p,q) : p,q\in \mathbb{R}\}.
\end{equation*}
Therefore, $\vartheta_{2^{k_j}}(x'-x-y)\psi_{2^l}(y'-x-y)$ can be written as
\[ \int_{\mathbb{R}^2} \vartheta_{2^{k_j}}(x'-p-q)\psi_{2^l}(y'-p-q) \omega_{2^l}(x-p)\omega_{2^l}(y-q) \,dp dq \]
and the object we need to bound is
\begin{align}\label{eq:formskipmain} \nonumber
& \sum_{j=1}^m \sum_{l=k_{j-1}}^{k_j-1} \int_{\mathbb{R}^6} F(y,x') G(x,x')F(y,y')G(x,y') \\
& \ \vartheta_{2^{k_j}}(x'-p-q) \psi_{2^l}(y'-p-q)\omega_{2^l}(x-p) \omega_{2^l}(y-q) \,dx dy dx' dy' dp dq.
\end{align}

In order to estimate this form we adapt the arguments from \cite{vk:dea} to the Euclidean setting. First we apply the Cauchy-Schwarz inequality, which will reduce the complexity of the form. To preserve the mean zero property of $\omega$ we rewrite \eqref{eq:formskipmain} as
\begin{align*}
\sum_{j=1}^m \sum_{l=k_{j-1}}^{k_j-1} \int_{\mathbb{R}^4} \Big( \int_{\mathbb{R}} F(y,x')F(y,y') \omega_{2^l}(y-q) \,dy \Big) \Big( \int_{\mathbb{R}} G(x,x')G(x,y') \omega_{2^l}(x-p) \,dx \Big) & \\
\vartheta_{2^{k_j}}(x'-p-q) \psi_{2^l}(y'-p-q) \,dx' dy' dp dq & .
\end{align*}
Taking absolute values, using the triangle inequality, and applying the Cauchy-Schwarz inequality in the variables $x'$, $y'$, $p$, $q$, and $t$, we bound this expression by
\begin{equation}\label{eq:csgamma}
\Gamma(F)^{1/2}\Gamma(G)^{1/2},
\end{equation}
where we have denoted
\begin{align*}
\Gamma(F) := \sum_{j=1}^m \sum_{l=k_{j-1}}^{k_j-1} \int_{\mathbb{R}^4} \Big( \int_{\mathbb{R}} F(y,x') F(y,y') \omega_{2^l}(y-q) dy \Big)^2 & \\
|\vartheta|_{2^{k_j}}(x'-p) |\psi|_{2^l}(y'-p) \,dx' dy' dp dq & .
\end{align*}
Here the two appearances of the function $\omega$ have been separated, which allowed us to change variables $p\rightarrow p-q$ in the last expression. Integrating in $p$, using $l \leq {k_j}$ and the normalization of $\vartheta$ and $\varphi$, we get
\begin{equation}\label{eq:mainest}
\int_{\mathbb{R}} |\vartheta|_{2^{k_j}}(x'-p) |\psi|_{2^l}(y'-p) dp \lesssim_\lambda {2^{-k_j}} (1+ {2^{-k_j}}{|x'-y'|} )^{-\lambda} .
\end{equation}
This fact can be shown along the lines of \cite[Lemma 2.1]{ct:wpa}. For completeness and to keep track of the constants we now give a detailed proof.

If $|x'-y'|\leq 2^{k_j+1}/(\lambda-1)$, then we can bound the left hand-side of \eqref{eq:mainest} by
\[ \|\vartheta_{2^{k_j}}\|_{\textup{L}^\infty(\mathbb{R})} \|\psi_{2^l}\|_{\textup{L}^1(\mathbb{R})} \lesssim_\lambda \|\vartheta\|_{\textup{L}^\infty(\mathbb{R})} \|\psi \|_{\textup{L}^1(\mathbb{R})}{{2^{-k_j}}} (1+ {2^{-k_j}}{|x'-y'|} )^{-\lambda}. \]
If $|x'-y'|\geq 2^{k_j+1}/(\lambda-1)$, then let us denote by $c$ the midpoint of $x'$ and $y'$. Without loss of generality we may assume $x'<c<y'$. We split the integral as $\int_\mathbb{R} = \int_{-\infty}^c + \int_c^\infty$ and estimate it by
\begin{equation}\label{eq:auxdisplay}
\|\vartheta \|_{\textup{L}^1(\mathbb{R})}{2^{-l}} (1+ 2^{-l}{|y'-c|})^{-\lambda} + {{2^{-k_j}}} (1+ {2^{-k_j}}{|x'-c|} )^{-\lambda} \|\psi \|_{\textup{L}^1(\mathbb{R})}.
\end{equation}
Since $|x'-c|=|y'-c|=|x'-y'|/2$,\, $l\leq k_j$,
\[ 2^{-l-1}|x'-y'|\geq 2^{-k_j-1}|x'-y'| \geq (\lambda-1)^{-1}, \]
and the function $s\mapsto s(1+s)^{-\lambda}$ is decreasing on the interval $[(\lambda-1)^{-1},\infty)$, the expression \eqref{eq:auxdisplay} is at most
\[ (\|\vartheta\|_{\textup{L}^1(\mathbb{R})}+\|\psi\|_{\textup{L}^1(\mathbb{R})}) \, 2^{-k_j} (1+ 2^{-k_j-1}{|x'-y'|})^{-\lambda} . \]
It remains to note $\|\vartheta\|_{\textup{L}^\infty(\mathbb{R})} \leq 1$, $\|\vartheta\|_{\textup{L}^1(\mathbb{R})} \lesssim_\lambda 1$, and $\|\psi\|_{\textup{L}^1(\mathbb{R})}\lesssim_\lambda 1$, which shows the claim.

Our inequality did not preserve the tensor structure in the variables $x'$ and $y'$ which will be needed later in \eqref{eq:tensor}. For that purpose we further estimate \eqref{eq:mainest} by a superposition of Gaussians as it was done in \cite{pd:L4}. Denote
\begin{equation}\label{eq:superpositiongauss}
g(s) := e^{- \pi s^2} \quad\text{and}\quad \sigma(s):=\int_1^\infty \!g_\alpha(s) \alpha^{-\lambda} d\alpha,
\end{equation}
where $g_\alpha(s)=\alpha^{-1}g(\alpha^{-1}s)$, as before. Observe that $\sigma(0)=\lambda^{-1}$ and the change of variables $\beta=|s|/\alpha$ gives
\[ \lim_{|s|\to\infty} |s|^\lambda\sigma(s) = \int_{0}^{\infty} \beta^{\lambda-1} e^{-\pi\beta^2} d\beta \in(0,\infty), \]
so $\sigma(s)$ is comparable to $|s|^{-\lambda}$ for large $|s|$. Therefore, using
\begin{equation}\label{eq:supgaussdomest}
(1+|s|)^{-\lambda} \lesssim_\lambda \sigma(s)
\end{equation}
we can dominate the right hand-side of \eqref{eq:mainest} up to a positive constant by $\sigma_{2^{k_j}}({x'-y'})$. This in turn controls
\begin{align}\label{eq:formpositive} \nonumber
\Gamma(F) \lesssim_\lambda \int_1^\infty \Big( & \sum_{j=1}^m \sum_{l=k_{j-1}}^{k_j-1} \int_{\mathbb{R}^5} F(y,x') F(x,x')F(y,y')F(x,y') \\
& g_{\alpha2^{k_j}}(x'-y') \omega_{2^l}(x-q) \omega_{2^l}(y-q) \,dx dy dx' dy' dq \, \Big) \alpha^{-\lambda} d\alpha.
\end{align}
Integrating in $q$, summing in $l$, and using $\widehat{\omega}(\xi)^2 = \sum_{i=-2}^3 \big(\widehat{\chi}(2^i\xi) - \widehat{\chi}(2^{i+1}\xi)\big)$ we obtain
\[ \sum_{l=k_{j-1}}^{k_j-1} \int_{\mathbb{R}} \omega_{2^l}(x-q)\omega_{2^l}(y-q) \,dq = \sum_{i=-2}^3 (\chi_{2^{k_{j-1}+i}}-\chi_{2^{k_j+i}})(x-y). \]
Inserting this into \eqref{eq:formpositive}, the integrand in $\alpha$ can be rewritten as
\begin{align*}
\sum_{i=-2}^3 \sum_{j=1}^m \int_{\mathbb{R}^4} & F(y,x') F(x,x')F(y,y')F(x,y') \\[-1ex]
& g_{\alpha2^{k_j}}(x'-y') (\chi_{2^{k_{j-1}+i}}-\chi_{2^{k_j+i}})(x-y) \,dx dy dx' dy'.
\end{align*}
It suffices to prove an estimate uniform in $\alpha$ for each summand corresponding to a fixed $i$ and then integrate in $\alpha$ and sum over $-2\leq i\leq 3$. For two functions $\tilde{\rho},\rho\in\mathcal{S}(\mathbb{R})$ define
\begin{align*}
\Theta_{\tilde{\rho},\rho}(F):= \sum_{j=1}^m \int_{\mathbb{R}^4} & F(y,x') F(x,x')F(y,y')F(x,y') \\[-1ex]
& \tilde{\rho}_{2^{k_j}}(x'-y') ({\rho}_{2^{k_{j-1}}}-{\rho}_{2^{k_j}})(x-y) \,dx dy dx' dy'.
\end{align*}
The needed estimate is a direct consequence of the following lemma applied with $\rho=\chi_{2^i}$.

\begin{lemma}\label{lemma:twisted}
For any real-valued $F\in \mathcal{S}(\mathbb{R}^2)$, real-valued $\rho\in \mathcal{S}(\mathbb{R})$ and $\alpha\in(0,\infty)$ we have
\begin{equation}\label{eq:lemmathetatoshow}
\Theta_{g_\alpha,\rho}(F)\lesssim_\rho \|F\|_{\textup{L}^4(\mathbb{R}^2)}^4,
\end{equation}
where $g(s)=e^{-\pi s^2}$.
\end{lemma}

\begin{proof}
Once again we normalize $F$ as in \eqref{eq:normalization}. The first step is an application of the telescoping identity. If we denote
\begin{align*}
\widetilde{\Theta}_{\tilde{\rho},\rho}(F):= \sum_{j=1}^m \int_{\mathbb{R}^4} & F(y,x') F(x,x')F(y,y')F(x,y') \\[-1ex]
& (\tilde{\rho}_{2^{k_j-1}}-\tilde{\rho}_{2^{k_j}})(x'-y') {\rho}_{2^{k_{j-1}}}(x-y) \,dx dy dx' dy'
\end{align*}
and for $t>0$ define the single-scale quantity
\[ \Xi_{\tilde{\rho},\rho,t}(F) := \int_{\mathbb{R}^4} F(y,x') F(x,x')F(y,y')F(x,y')\tilde{\rho}_t(x'-y')\rho_t(x-y) \,dx dy dx' dy', \]
then we have
\begin{align}
\Theta_{\tilde{\rho},\rho}(F) + \widetilde{\Theta}_{\tilde{\rho},\rho}(F) & = \Xi_{\tilde{\rho},\rho, 2^{k_0}}(F) - \Xi_{\tilde{\rho},\rho, 2^{k_m}}(F), \label{eq:tel1} \\
\Xi_{\tilde{\rho},\rho,t}(F) & \leq \|\tilde{\rho}\|_{\textup{L}^1(\mathbb{R})}\|\rho\|_{\textup{L}^1(\mathbb{R})}. \label{eq:tel2}
\end{align}
The identity \eqref{eq:tel1} follows from summation by parts: all intermediate terms cancel. To see \eqref{eq:tel2} we substitute $u=x'-y'$, $v=x-y$, rewrite $\Xi_{\tilde{\rho},\rho,t}(F)$ as
\[ \int_{\mathbb{R}^2} \Big( \int_{\mathbb{R}^2} F(x-v,x')F(x,x')F(x-v,x'-u)F(x,x'-u) \,dx dx' \Big) \,\tilde{\rho}_t(u)\rho_t(v) \,du dv, \]
and apply H\"older's inequality in $(x,x')$ for the exponents $(4,4,4,4)$.

In order to show \eqref{eq:lemmathetatoshow} we first use \eqref{eq:tel1}, which gives
\begin{equation*}
\Theta_{g_\alpha,\rho}(F) = \Xi_{g_\alpha,\rho,2^{k_0}}(F) - \Xi_{g_\alpha, \rho,2^{k_m}}(F) - \widetilde{\Theta}_{g_\alpha,\rho}(F),
\end{equation*}
and hence applying \eqref{eq:tel2} we get
\[ |\Theta_{g_\alpha,\rho}(F)| \leq |\Xi_{g_\alpha,\rho,2^{k_0}}(F)| + |\Xi_{g_\alpha,\rho,2^{k_m}}(F)| + |\widetilde{\Theta}_{g_\alpha,\rho}(F)| \lesssim_\rho 1 + \big|\widetilde{\Theta}_{g_\alpha,\rho}(F)\big|. \]
Therefore, it remains to estimate $\big|\widetilde{\Theta}_{g_\alpha,\rho}(F)\big|$.

By the fundamental theorem of calculus we rewrite $\widetilde{\Theta}_{g_\alpha,\rho}(F)$ as
\begin{align*}
\widetilde{\Theta}_{g_\alpha,\rho}(F) = \sum_{j=1}^m \int_{2^{k_{j-1}}}^{2^{k_j}}\int_{\mathbb{R}^4} & F(y,x') F(x,x')F(y,y')F(x,y') \\[-1ex]
& \big(-t\partial_t(g_{\alpha t}(x'-y'))\big) \rho_{2^{k_{j-1}}}(x-y) \,dx dy dx' dy' \frac{dt}{t}.
\end{align*}
For $h(s) := \sqrt{2/\pi } g'(\sqrt{2}s)$ we have $-t\partial_t (\widehat{g_{\alpha t }}(\xi)) = |\widehat{h_{\alpha t}}(\xi)|^2$ and hence
\begin{equation}\label{eq:tensor}
-t\partial_t(g_{\alpha t}(x'-y')) = \int_{\mathbb{R}} h_{\alpha t}(x'-p) h_{\alpha t}(y'-p) dp.
\end{equation}
By this identity and the symmetry of $\widetilde{\Theta}_{g_\alpha,\rho}$, which results from four repetitions of the function $F$, we can express $\widetilde{\Theta}_{g_\alpha,\rho}(F)$ as
\begin{equation}\label{eq:formnonneg}
\sum_{j=1}^m \int_{{2^{k_{j-1}}}}^{2^{k_j}} \int_{\mathbb{R}^3} \Big(\int_{\mathbb{R}} F(y,x')F(x,x')h_{{\alpha}t}(x'-p) dx'\Big)^2 \rho_{2^{k_{j-1}}}(x-y) \,dx dy dp \frac{dt}{t}.
\end{equation}
Observe that the square in \eqref{eq:formnonneg} is automatically non-negative, but the function $\rho$ is not non-negative in general. To obtain positivity and an elementary tensor structure in $x$ and $y$ as in \eqref{eq:tensor} we dominate $|\rho|\lesssim\sigma$ by applying \eqref{eq:supgaussdomest} as before, where $\sigma$ is the superposition of the Gaussians \eqref{eq:superpositiongauss}. This implies
\[ \big|\widetilde{\Theta}_{g_\alpha,\rho}(F)\big| \lesssim_\rho \int_1^\infty \widetilde{\Theta}_{g_{\alpha},g_\beta}(F) {\beta^{-\lambda}}d\beta . \]
We apply the telescoping identity \eqref{eq:tel1} once more to get
\[ \widetilde{\Theta}_{g_{\alpha},g_\beta} (F) = \Xi_{g_{\alpha},g_\beta,2^{k_0}}(F) - \Xi_{g_{\alpha},g_\beta,2^{k_m}}(F) - \Theta_{g_{\alpha},g_\beta}(F). \]
Now that we have reduced to Gaussian functions only, we have non-negativity of both $\Theta_{g_{\alpha},g_\beta}(F)$ and $\widetilde{\Theta}_{g_{\alpha},g_\beta}(F)$. This can be seen by the fundamental theorem of calculus and the equality \eqref{eq:tensor}, which allow us to write $\Theta_{g_{\alpha},g_\beta}(F)$ and $\widetilde{\Theta}_{g_{\alpha},g_\beta}(F)$ in the same way as we did with the form in \eqref{eq:formnonneg}. Therefore, by \eqref{eq:tel2} once again,
\[ \widetilde{\Theta}_{g_{\alpha},g_\beta}(F) \leq \Xi_{g_\alpha,g_\beta,2^{k_0}}(F) - \Xi_{g_\alpha,g_\beta,2^{k_m}}(F)
\leq 2\|g\|_{\textup{L}^1(\mathbb{R})}^2 \lesssim 1. \]
This finishes the proof of Lemma~\ref{lemma:twisted}.
\end{proof}


\section{Proof of Proposition~\ref{prop:not-skip}}
\label{sec:propnotskip}
We would like to decompose the kernel of the form appearing on the left hand side of \eqref{eq:estnontensor} into elementary tensors analogous to those from Section~\ref{sec:propskip}. Then we could bound this form by the Cauchy-Schwarz inequality and iterations of the telescoping identity and positivity arguments. However, the multiplier support now intersects the axis $\eta=-\xi$, so a desired decomposition is not readily available.

To overcome this issue, the idea is to transfer to the multiplier with the symbol \eqref{eq:homogeneousmult} below, which is homogeneous, i.e.\@ constant on the rays through the origin, symmetric with respect to $\eta=-\xi$, and smooth away from that axis. Since the form with a constant multiplier is trivially bounded, we can then subtract the constant on $\eta=-\xi$ from that homogeneous multiplier. This leaves us with a function vanishing on $\eta=-\xi$ up to a certain positive order. By a bi-parameter lacunary decomposition with respect to the axes $\eta=\xi$ and $\eta=-\xi$ we reduce to the consideration of certain angular regions to which the arguments analogous to those from Section~\ref{sec:propskip} may be applied. Due to the vanishing along $\eta=-\xi$ we are able to sum over all such regions.

We start with a lemma which considers multiplier symbols supported away from the axis $\eta=-\xi$. It will be applied several times in the proof of Proposition~\ref{prop:not-skip}.

\begin{lemma}\label{lemma:nontensor}
Let $\lambda>1$, $t>0$ and let $\Phi \in \mathcal{S}(\mathbb{R}^2)$ be such that
\[ |\Phi(u,v)| \leq (1+|u+v|)^{-\lambda}t(1+t|u-v|)^{-\lambda}. \]
for all $u,v\in \mathbb{R}$.
Moreover, assume that $2^{-2} \leq |\xi+\eta| \leq 1$ for all $(\xi,\eta)$ in the support of $\widehat{\Phi}$. Then for any real-valued $F,G\in \mathcal{S}(\mathbb{R}^2)$ normalized as in \eqref{eq:normalization} and for any $N\in\mathbb{N}$ we have \eqref{eq:estnontensor}.
\end{lemma}

\begin{proof}[Proof of Lemma \ref{lemma:nontensor}]
Our aim is to reduce Lemma~\ref{lemma:nontensor} to Lemma~\ref{lemma:twisted} from the previous section. Let $\chi$ and $\omega$ be the functions as in Section~\ref{sec:propskip}. Then $\widehat{\omega}(\xi+\eta)$ equals $1$ on $\{\xi+\eta : (\xi,\eta)\in \mathrm{supp}(\widehat{\Phi})\}$, so for each $(\xi,\eta)\in \mathbb{R}^2$ we can write
\[ \widehat{\Phi}(\xi,\eta) = \widehat{\Phi} (\xi,\eta) \widehat{\omega}(\xi+\eta)^2, \]
similarly as in \eqref{eq:addomega}. Choosing the same substitution as in Section~\ref{sec:propskip} and performing the analogous steps from \eqref{eq:addomega} to \eqref{eq:csgamma} with $k_j$ and $l$ being replaced by $j$, it remains to estimate an analogous quantity to $\Gamma(F)$,
\begin{align*}
\sum_{j=-N}^{N} \int_{\mathbb{R}^6} & F(y,x')F(x,x')F(y,y')F(x,y') \\[-1ex]
& |\Phi|_{2^j}(x'-p,y'-p) \omega_{2^j}(x-q)\omega_{2^j}(y-q) \,dx dy dx' dy' dp dq.
\end{align*}
Using the decay assumption on $\Phi$ we obtain
\begin{align*}
\int_{\mathbb{R}}|\Phi|(x'-p,y'-p) \,dp & \leq \int_{\mathbb{R}} t(1+t|x'-y'|)^{-\lambda} (1+ |x'+y'-2p|)^{-\lambda}dp \\
& \lesssim_\lambda t(1+t|x'-y'|)^{-\lambda}.
\end{align*}
Estimating the right hand-side as in \eqref{eq:supgaussdomest} by the superposition $\sigma$ defined in \eqref{eq:superpositiongauss} and proceeding as we did with \eqref{eq:formpositive}, it then suffices to bound
\[ \sum_{j=-N}^{N} \int_{\mathbb{R}^4} F(y,x')F(x,x')F(y,y')F(x,y') g_{\alpha t2^j}(x'-y') (\chi_{2^{j+i}}-\chi_{2^{j+i+1}})(x-y) \,dx dy dx' dy' \]
uniformly in $\alpha,t\in(0,\infty)$ and for each fixed $-2\leq i\leq 3$. Such an estimate follows from the particular case of Lemma~\ref{lemma:twisted} when $\rho=\chi_{2^{i+1}}$ and $k_j=j$.
\end{proof}

Now we are ready to proceed with the proof of Proposition \ref{prop:not-skip}.
We can assume that $1<\lambda<2$, as the claim only becomes stronger as $\lambda$ decreases to $1$. Recall that the form from Proposition~\ref{prop:not-skip} is associated with the kernel
\[ K(u,v) := \sum_{j=-N}^{N}{\Phi}_{2^j}(u,v). \]
Let $\theta$ be $\chi-\chi_2$, so that $\widehat{\theta}$ partitions the unity as in \eqref{eq:partunity}. Then $\int_0^\infty \widehat{\theta}(t\tau) \frac{dt}{t}$ is the same constant for all $0\neq\tau\in\mathbb{R}$ and up to that constant $\widehat{K}(\xi,\eta)$ equals
\begin{equation}\label{eq:mexpanded}
\int_0^\infty \widehat{K}(\xi,\eta)\widehat{\theta}(t|(\xi,\eta)|) \frac{dt}{t} = \int_0^\infty \widehat{K^{(t)}}(t(\xi,\eta)) \frac{dt}{t}
\end{equation}
for all $(\xi,\eta)\neq (0,0)$, where $K^{(t)}$ is defined via its Fourier transform as
\[ \widehat{K^{(t)}}(\xi,\eta):=\widehat{K}(t^{-1}(\xi,\eta))\widehat{\theta}(|(\xi,\eta)|). \]
Observe that the support of $\widehat{K^{(t)}}(\xi,\eta)$ lies in the intersection of the annulus $2^{-2}\leq|(\xi,\eta)|\leq 1$ with the quadruple cone $2^{-6}\leq |\eta/\xi| \leq 2^6$, which in turn is contained in the Cartesian product
\begin{equation}\label{eq:suppset}
([-1,-2^{-9}]\cup [2^{-9},1])^2.
\end{equation}
Let $\vartheta$ be such that $\widehat{\vartheta}$ is a smooth nonnegative even function supported in $[-2,-2^{-10}]\cup [2^{-10},2]$ and such that $(\xi,\eta)\mapsto \widehat{\vartheta}(\xi)\widehat{\vartheta}(\eta)$ equals $1$ on the set \eqref{eq:suppset} and thus also on the support of each $\widehat{K^{(t)}}$. Then
\[ \widehat{K^{(t)}}(\xi,\eta) = \widehat{K^{(t)}}(\xi,\eta) \widehat{\vartheta}(\xi) \widehat{\vartheta}(\eta), \]
which implies
\begin{equation}\label{eq:mexpanded2}
K^{(t)}(u,v) = \int_{\mathbb{R}^2} K^{(t)}(a,b) \vartheta(u-a) \vartheta(v-b) \,da db.
\end{equation}
Using \eqref{eq:mexpanded} and \eqref{eq:mexpanded2}, the form from Proposition~\ref{prop:not-skip} can be rewritten as
\begin{align}
\int_{\mathbb{R}^2} \int_0^\infty K^{(t)}(a,b) \int_{\mathbb{R}^4} F(x+u,y)G(x,y+u)F(x+v,y)G(x,y+v) & \nonumber \\[-1ex]
\vartheta_t(u-ta) \vartheta_t(v-tb) \,dx dy du dv \,\frac{dt}{t} \,da db & . \label{eq:formnotskipdec}
\end{align}
Observe that for $\kappa\in\mathcal{S}(\mathbb{R}^2)$ defined by $\widehat{\kappa}(\xi,\eta)=\widehat{\theta}(|(\xi,\eta)|)$ we have
\[ K^{(t)}(a,b) = \sum_{j=-N}^{N} \int_{\mathbb{R}^2} \Phi_{2^j/t}(a-x, b-y) {\kappa}(x,y) \,dxdy \]
and by the support conditions on $\widehat{\Phi}$ and $\widehat{\kappa}$ the sum is taken only over $-N\leq j\leq N$ that also satisfy $2^{-7}<2^j/t<2^7$. Thus, there are at most $14$ non-zero summands for each fixed $t$ and $2^j/t\sim1$ holds for each of them. From the assumption \eqref{eq:cdprop6} transformed into \eqref{eq:cdprop6equiv} and the rapid decay of $\kappa$ it follows that
\begin{align*}
\big|K^{(t)}(a,b)\big| & \lesssim_\lambda \int_{\mathbb{R}^2} (1+|a-x|)^{-\lambda/2}(1+|b-y|)^{-\lambda/2} (1+|a-x -b+y|)^{-\lambda}|{\kappa}(x,y)| \,dxdy \\
& \lesssim_\lambda (1+|a|)^{-\lambda/2}(1+|b|)^{-\lambda/2} (1+|a-b|)^{-\lambda}.
\end{align*}

Taking absolute values in \eqref{eq:formnotskipdec} and denoting
\[ I(x,y,a,t) := \int_\mathbb{R} F(x+s,y)G(x,y+s) \vartheta_t(s-ta) ds, \]
we can now bound \eqref{eq:formnotskipdec} by
\[ \int_{\mathbb{R}^2} (1+|a|)^{-\lambda/2}(1+|b|)^{-\lambda/2} (1+|a-b|)^{-\lambda} \int_0^\infty \!\int_{\mathbb{R}^2} |I(x,y,a,t)I(x,y,b,t)| \,dx dy \,\frac{dt}{t} \,dadb. \]
Next, we apply the Cauchy-Schwarz inequality in $x,y$ and $t$, which gives
\begin{align}
\int_{\mathbb{R}^2} (1+|a-b|)^{-\lambda} (1+|a|)^{-\lambda/2} \Big( \int_{0}^{\infty} \!\int_{\mathbb{R}^2} I(x,y,a,t)^2 \,dx dy \,\frac{dt}{t} \Big)^{1/2} & \nonumber \\
(1+|b|)^{-\lambda/2} \Big( \int_{0}^{\infty} \!\int_{\mathbb{R}^2} I(x,y,b,t)^2 \,dx dy \,\frac{dt}{t} \Big)^{1/2} & da db. \label{eq:sec6aftercs}
\end{align}
If we denote
\[ J(a) := (1+|a|)^{-\lambda/2}\Big( \int_{0}^{\infty} \!\int_{\mathbb{R}^2} I(x,y,a,t)^2 \,dx dy \,\frac{dt}{t} \Big)^{1/2}, \]
the expression \eqref{eq:sec6aftercs} can be rewritten as
\[\int_\mathbb{R} \Big( \int_\mathbb{R} (1+|a-b|)^{-\lambda} J(a) da \Big) J(b)db. \]
Applying the Cauchy-Schwarz inequality in $b$ we obtain
\begin{equation}\label{eq:sec6aftercs2}
 \Big( \int_\mathbb{R} \Big( \int_\mathbb{R} (1+|a-b|)^{-\lambda} J(a) da \Big)^2 db\Big)^{1/2} \Big( \int_{\mathbb{R}} J(b)^2 db \Big)^{1/2} .
\end{equation}
Note that the integral in $a$ is the convolution of $J$ with $s\mapsto (1+|s|)^{-\lambda}$. By Young's convolution inequality from $\textup{L}^1(\mathbb{R})\times\textup{L}^2(\mathbb{R})$ to $\textup{L}^2(\mathbb{R})$, the expression \eqref{eq:sec6aftercs2} is bounded by a constant multiple of
\[ \|J\|^2_{\textup{L}^2(\mathbb{R})} \leq \int_{\mathbb{R}} (1+a^2)^{-\lambda/2} \int_{0}^{\infty} \!\int_{\mathbb{R}^2} I(x,y,a,t)^2 \,dx dy \,\frac{dt}{t} da. \]
Expanding $I$, this equals
\begin{align}
\int_\mathbb{R} (1+a^2)^{-\lambda/2} \int_0^\infty \!\int_{\mathbb{R}^4} F(x+u,y)G(x,y+u)F(x+v,y)G(x,y+v) & \nonumber \\[-1ex]
\vartheta_t(u-ta) \vartheta_t(v-ta) \,dx dy du dv \frac{dt}{t} da & . \label{eq:formnotskipsym}
\end{align}
Observe that this form is associated with the multiplier symbol
\begin{equation}\label{eq:homogeneousmult}
M(\xi,\eta) := \int_0^\infty \widehat{\vartheta}(t\xi)\widehat{\vartheta}(t\eta)\widehat{\rho}(t(\xi+\eta)) \frac{dt}{t},
\end{equation}
where we have denoted
\begin{equation}\label{eq:rhojap}
\rho(s) := (1+s^2)^{-\lambda/2}.
\end{equation}
Note that the function $\widehat{\rho}$ is even and hence $M(\xi,\eta)=M(-\eta,-\xi)$. Moreover, $M$ is constant on any line through the origin and in particular $M(\xi,-\xi)=M(1,-1)$ for any $0\neq \xi\in \mathbb{R}$.
Now we write
\[ M(\xi,\eta) = M(1,-1) + \big( M(\xi,\eta)-M(1,-1) \big) \]
and split the form \eqref{eq:formnotskipsym} into the two corresponding parts. The part associated with the constant multiplier yields $M(1,-1)$ times
\begin{align*}
\int_{\mathbb{R}^4} F(x+u,y)G(x,y+u)F(x+v,y)G(x,y+v) \delta_{(0,0)}(u,v) \,dx dy du dv & \\
= \int_{\mathbb{R}^4} F(x,y)^2 G(x,y)^2 \,dx dy \leq \|F\|_{\textup{L}^4(\mathbb{R}^2)}^2 \|G\|_{\textup{L}^4(\mathbb{R}^2)}^2 = 1 & ,
\end{align*}
where $\delta_{(0,0)}$ denotes the Dirac measure concentrated at the origin. Thus, our remaining task is to estimate the form associated with the symbol $M_0:=M-M(1,-1)$.

For each   $(\xi,\eta)\in \mathbb{R}^2\setminus \{(\xi,\eta) : \xi=\eta \; \textup{or}\; \xi=-\eta\}$ we  decompose
\begin{equation} \label{eq:m0decompose}
M_{0}(\xi,\eta) = \sum_{k\in\mathbb{Z}}\sum_{j\in\mathbb{Z}} M_{0}(\xi,\eta)\widehat{\theta}(2^{j+k}(\xi-\eta))\widehat{\theta}(2^{j}(\xi+\eta)).
\end{equation}
If we denote
\[ m^{(k)}(\xi,\eta):=M_{0}(\xi,\eta)\widehat{\theta}(2^{k}(\xi-\eta))\widehat{\theta}(\xi+\eta)  \]
and
\[ {m}(\xi,\eta):= \sum_{k\geq 0}m^{(k)}(\xi,\eta) = M_0(\xi,\eta) \Big(\sum_{k\geq 0}\widehat{\theta}(2^{k}(\xi-\eta)) \Big )\widehat{\theta}(\xi+\eta)   , \]
and split the summation in \eqref{eq:m0decompose} over the regions $k\geq 0$ and $k<0$, we obtain
\begin{equation*}
M_{0} (\xi,\eta) = \sum_{k\in \mathbb{Z}}\sum_{j\in \mathbb{Z}}m^{(k)}(2^{j}(\xi,\eta)) = \sum_{j\in \mathbb{Z}} {m}(2^j(\xi,\eta)) + \sum_{k<0}\sum_{j\in \mathbb{Z}}m^{(k)}(2^{j}(\xi,\eta)).
\end{equation*}
Here we used that $M_0(\xi,\eta)=M_0(2^j(\xi,\eta))$ by homogeneity.

First we treat the form associated with the multiplier symbol
\begin{equation}\label{symbolfirst}
\sum_{j\in \mathbb{Z}} {m}(2^j(\xi,\eta)).
\end{equation}
Observe that $ {m}$ is compactly supported in the strip $2^{-2}\leq|\xi+\eta|\leq1$.
Moreover, we have
\begin{equation}\label{eq:tildemdecay}
|\widecheck{{m}}(u,v)|\lesssim_\lambda (1+|u+v|)^{-\lambda} (1+|u-v|)^{-2}.
\end{equation}
Indeed, this estimate can be seen by bounding the inverse Fourier transform of
\begin{equation} \label{boundft}
(\xi,\eta)\mapsto  {m}(\xi,\eta) + M(1,-1)\phi(\xi,\eta),
\end{equation}
where we have set
\[ \phi(\xi,\eta):= \Big(\sum_{k\geq 0}\widehat{\theta}(2^{k}(\xi-\eta)) \Big )\widehat{\theta}(\xi+\eta)   . \]
Therefore, the inverse Fourier transform of \eqref{boundft} is nothing but
\[ (u,v)\mapsto \int_{\mathbb{R}} \int_{0}^{\infty} \rho(a) \vartheta_t(u-ta) \vartheta_t(v-ta) \frac{dt}{t} da \]
convolved with the Schwartz function $\widecheck{\phi}$ and by the support localization of $\phi$ we may assume that  $t$ ranges over a fixed bounded subinterval of $(0,\infty)$. It remains to observe that
\begin{align*}
\Big| \int_{\mathbb{R}}\rho(a)\vartheta_t(u-ta)\vartheta_t(v-ta) da \Big| & \lesssim_{\vartheta,\lambda} (1+|u-v|)^{-2} \int_{\mathbb{R}}\rho(a)(1+|u+v-2a|)^{-2\lambda} da \\
& \lesssim_\lambda (1+|u-v|)^{-2}(1+|u+v|)^{-\lambda}.
\end{align*}
which in turn implies \eqref{eq:tildemdecay}.
Boundedness of the form associated with \eqref{symbolfirst} now follows from Lemma~\ref{lemma:nontensor} applied with $\Phi=\widecheck{{m}}$ and by letting $N\to\infty$.

 It remains to consider the form associated  with the symbol
\begin{equation}\label{eq:multiplierneq}
\sum_{k<0}\sum_{j\in \mathbb{Z}}m^{(k)}(2^j(\xi,\eta)).
\end{equation}
Note that $m^{(k)}$ is supported in the strip $2^{-2}\leq|\xi+\eta|\leq1$ for each $k$.
To estimate the form associated with \eqref{eq:multiplierneq} it now suffices to show that for each $k<0$ we have
\begin{equation}\label{eq:multiplierdecay3}
|\widecheck{m^{(k)}}(u,v)|\lesssim_\lambda 2^{k(\lambda-1)} \,(1+|u+v|)^{-2} \,2^{-k}(1+2^{-k}|u-v|)^{-2},
\end{equation}
with the implicit constant independent of $k$. Once we have that, boundedness of the form associated with the symbol in \eqref{eq:multiplierneq} for a fixed $k$ follows from Lemma~\ref{lemma:nontensor} applied with $\Phi = \widecheck{m^{(k)}}$ and by letting $N\rightarrow \infty$.
In the end it remains to sum the geometric series: $\sum_{k<0} 2^{k(\lambda-1)} \lesssim_\lambda 1$.

The estimate \eqref{eq:multiplierdecay3} will be deduced by integration by parts in the Fourier expansion of $m^{(k)}$ once we verify the necessary symbol estimates. At this point we switch to the frequency coordinates $\xi-\eta$ and $\xi+\eta$, which are better suited for our problem.
First, we claim that for any $0\leq n \leq 2$, $|\alpha|\sim 1$, and $0<|\beta|\leq 1$ we have
\begin{equation}\label{eq:symbolestm0}
\big|\partial_\beta \partial_\alpha^n \big(M_0 (\alpha+\beta,\beta-\alpha)\big)\big| \lesssim_\lambda |\beta|^{\lambda-2},\quad
\big|\partial_\beta^{2} \partial_\alpha^n \big(M_0 (\alpha+\beta,\beta-\alpha)\big)\big| \lesssim_\lambda |\beta|^{\lambda-3}.
\end{equation}
For now let us assume that the estimates in \eqref{eq:symbolestm0} hold. For $0 \leq n \leq 2$ define
\[ \mu^{(n)}(\alpha,\beta):= \partial_\alpha^n \big(M_0(\alpha+\beta, \beta-\alpha)\big)  \]
and note that  $\mu^{(n)}(\alpha,0) = 0$. Therefore, for any $|\alpha|\sim 1$, $0<|\beta|\leq 1$, and $0\leq n \leq 2$, the first estimate in  \eqref{eq:symbolestm0} implies
\begin{equation}\label{eq:integratederivative}
\big|\partial_\alpha^n \big(M_0(\alpha+\beta, \beta-\alpha)\big)\big| = \Big| \int_0^\beta \partial_2\mu^{(n)}(\alpha,\gamma)d\gamma \Big| \lesssim_\lambda |\beta|^{\lambda-1}.
\end{equation}
The estimates \eqref{eq:symbolestm0} and \eqref{eq:integratederivative} together imply that for any $0\leq l,n \leq 2$ one has
\begin{align*}
\big|\partial_\beta^{l}\partial_\alpha^{n}\big( M_0( \alpha+2^{k}\beta, 2^{k}\beta-\alpha)\widehat{\theta}(2\alpha)\widehat{\theta}(2\beta)\big)\big| \lesssim_{\lambda} 2^{k(\lambda-1)},
\end{align*}
which is by the homogeneity of $M_0$ equivalent to
\begin{align}\label{eq:symbolestk}
\big|\partial_\beta^{l}\partial_\alpha^{n} \big(m^{(k)}(2^{-k}\alpha+\beta,\beta-2^{-k}\alpha)\big)\big| \lesssim_{\lambda} 2^{k(\lambda-1)} .
\end{align}

We proceed by verifying \eqref{eq:multiplierdecay3}.
Let us write
\[ u\xi+v\eta = 2^{-k}(u-v)2^{k-1}(\xi-\eta) + (u+v)2^{-1}(\xi+\eta). \]
Changing variables $(\alpha,\beta)=(2^{k-1}(\xi-\eta),2^{-1}(\xi+\eta))$ gives
\begin{align*}
\widecheck{m^{(k)}} (u,v) & = \int_{\mathbb{R}^2} m^{(k)}(\xi,\eta) e^{2\pi i (u\xi+v\eta)} d\xi d\eta \\
& = 2^{-k+1}\int_{\mathbb{R}^2} m^{(k)}(2^{-k}\alpha+\beta, \beta-2^{-k}\alpha)e^{2\pi i (2^{-k}(u-v)\alpha + (u+v)\beta) } d\alpha d\beta.
\end{align*}
If $|u-v|\leq 2^{k}$ and $|u+v|\leq 1$, then we bound
\[ |\widecheck{m^{(k)}} (u,v)| \lesssim 2^{-k} \|{m}^{(k)}\|_{\textup{L}^{\infty}(\mathbb{R}^2)}
\lesssim_{\lambda} 2^{k(\lambda-1)}\, 2^{-k}, \]
which implies \eqref{eq:multiplierdecay3} in this case.
Here we used \eqref{eq:symbolestk} to control the $\textup{L}^\infty$ norm and observed
\begin{equation}
\label{eq:supportmk}
\big\{ (\alpha,\beta) : m^{(k)}(2^{-k}\alpha+\beta, \beta-2^{-k}\alpha)\neq 0 \big\} \subseteq ([-2^{-1},-2^{-3}]\cup[2^{-3},2^{-1}])^2.
\end{equation}
Now assume that $|u-v|\geq 2^{k}$ and $|u+v|\geq 1$.
Integrating by parts we bound $|\widecheck{m^{(k)}} (u,v)|$ by a constant multiple of
\[ 2^{-k} (2^{-k}|u-v|)^{-2} |u+v|^{-2} \Big| \int_{\mathbb{R}^2} \partial_{\beta}^2\partial^2_{\alpha} ({m}^{(k)}(2^{-k}\alpha+\beta, \beta-2^{-k}\alpha)) e^{2\pi i (2^{-k}(u-v)\alpha + (u+v)\beta)} d\alpha d\beta \Big|. \]
Together with \eqref{eq:symbolestk} and \eqref{eq:supportmk} this shows \eqref{eq:multiplierdecay3} in the present case. If $|u-v|\geq 2^{k}$ and $|u+v|\leq 1$, or vice versa, we simply  combine the arguments from both of the discussed cases.

It remains to show \eqref{eq:symbolestm0} and for that we need
\begin{align}\label{eq:rhoest}
|\widehat{\rho}'(\xi)| \lesssim_{\lambda} |\xi|^{\lambda-2}, \quad
|\widehat{\rho}''(\xi)| \lesssim_{\lambda} |\xi|^{\lambda-3}
\end{align}
for $|\xi|\leq 1$, where $\rho$ is our very particular choice of function \eqref{eq:rhojap}. The following formulae that hold for $\xi>0$ can be found using \cite{w:m9} or \cite{as:hmf}:
\begin{align*}
\widehat{\rho}(\xi) & = 2\pi^{\lambda/2}\xi^{(\lambda-1)/2} \mathcal{K}_{(1-\lambda)/2}(2\pi\xi) / \Gamma(\lambda/2), \\
\widehat{\rho}'(\xi) & = -4\pi^{1+\lambda/2}\xi^{(\lambda-1)/2} \mathcal{K}_{(\lambda-3)/2}(2\pi\xi) / \Gamma(\lambda/2), \\
\widehat{\rho}''(\xi) & = 4\pi^{1+\lambda/2}\xi^{(\lambda-3)/2} \big( 2\pi\xi \mathcal{K}_{(\lambda-5)/2}(2\pi\xi) - \mathcal{K}_{(\lambda-3)/2}(2\pi\xi) \big) / \Gamma(\lambda/2),
\end{align*}
where $\mathcal{K}_\alpha$ is the modified Bessel function of the second kind, given for $\alpha\not\in\mathbb{Z}$ and $z>0$ by the series
\[ \mathcal{K}_\alpha(z) = \frac{\pi}{2\sin(\alpha\pi)} \bigg(\sum_{n=0}^{\infty}\frac{1}{n!\Gamma(n-\alpha+1)}\Big(\frac{z}{2}\Big)^{2n-\alpha} - \sum_{n=0}^{\infty}\frac{1}{n!\Gamma(n+\alpha+1)}\Big(\frac{z}{2}\Big)^{2n+\alpha}\bigg). \]
From this expansion we read off the asymptotic behaviors in a neighborhood of $0$:
\[ |\mathcal{K}_{\alpha}(z)| \sim_\alpha z^{\min\{\alpha,-\alpha\}},\quad |\widehat{\rho}'(\xi)| \sim_\lambda |\xi|^{\lambda-2},\quad |\widehat{\rho}''(\xi)| \sim_\lambda |\xi|^{\lambda-3}, \]
which establish \eqref{eq:rhoest}.
Alternatively, to obtain these estimates one could decompose $\widehat{\rho}$ into the Littlewood-Paley pieces and argue by scaling.
Finally, differentiation of $M(\alpha+\beta,\beta-\alpha)$ using \eqref{eq:rhoest} and the product rule gives \eqref{eq:symbolestm0}.


\section{Ergodic averages, deriving Theorem~\ref{thm:ergodicthm} from Theorem~\ref{thm:analyticthmRough}}
\label{sec:transition}
Take $m\in\mathbb{N}$ and arbitrary positive integers $n_0<n_1<\cdots<n_m$.
For $F,G\in \textup{L}^4(\mathbb{R}^2)$
denote $A_t(F,G) := A_t^{\mathbbm{1}_{[0,1)}}(F,G)$, so that
\begin{align}
A_t(F,G)(x,y) = & \,\frac{1}{t}\int_{[0,t)} F(x+s,y) \,G(x,y+s) \,ds \nonumber \\
= & \,\frac{1}{t}\int_{[x+y,x+y+t)} F(u-y,y) \,G(x,u-x) \,du. \label{eq:avsubexp}
\end{align}
Applying Theorem~\ref{thm:analyticthmRough} to the scales $t_j=n_j$ and arbitrary functions $F,G\in\textup{L}^4(\mathbb{R}^2)$ normalized as in \eqref{eq:normalization} gives
\begin{equation}\label{eq:applytheorem2}
\sum_{j=1}^{m} \| A_{n_j}(F,G) - A_{n_{j-1}}(F,G) \|_{\textup{L}^2(\mathbb{R}^2)}^2 \lesssim 1 .
\end{equation}

Now we transfer the obtained estimate from $\mathbb{R}^2$ to $\mathbb{Z}^2$. Recall the definition \eqref{eq:defineatilde} of the averages $\widetilde{A}_n$ and observe that they can be rewritten as
\begin{equation}
\widetilde{A}_n(\widetilde{F},\widetilde{G})(k,l) = \frac{1}{n}\sum_{\substack{i\in\mathbb{Z}\\ k+l\leq i\leq k+l+n-1}} \widetilde{F}(i-l,l) \,\widetilde{G}(k,i-k). \label{eq:auxonZ}
\end{equation}
Pick arbitrary $\widetilde{F},\widetilde{G}\in\ell^4(\mathbb{Z}^2)$ normalized by $\|\widetilde{F}\|_{\ell^4(\mathbb{Z}^2)} = \|\widetilde{G}\|_{\ell^4(\mathbb{Z}^2)} =1$. Define the functions $F,G\colon\mathbb{R}^2\to\mathbb{R}$ as
\begin{align*}
F(x,y) & := \sum_{i,l\in\mathbb{Z}} \widetilde{F}(i-l,l) \,\mathbbm{1}_{[i,i+1)}(x+y) \,\mathbbm{1}_{[l,l+1)}(y), \\
G(x,y) & := \sum_{i,k\in\mathbb{Z}} \widetilde{G}(k,i-k) \,\mathbbm{1}_{[k,k+1)}(x) \,\mathbbm{1}_{[i,i+1)}(x+y).
\end{align*}
Note that $F$ and $G$ are constant on certain skew parallelograms of area $1$ and $\|F\|_{\textup{L}^4(\mathbb{R}^2)} = \|G\|_{\textup{L}^4(\mathbb{R}^2)} = 1$ as well.
Splitting the integral \eqref{eq:avsubexp} into the pieces over $i\leq u<i+1$ we get
\begin{equation}
\label{eq:auxonR}
A_n(F,G)(k+\alpha,l+\beta) = \frac{1}{n}\sum_{i\in\mathbb{Z}} a_i \,\widetilde{F}(i-l,l) \,\widetilde{G}(k,i-k),
\end{equation}
for any $k,l\in\mathbb{Z}$, $\alpha,\beta\in[0,1)$, where we have denoted
\[ a_i = \big|[i,i+1)\cap[k+l+\alpha+\beta,k+l+\alpha+\beta+n)\big|. \]
Observe that
\[ \begin{array}{rl} a_i=1 & \text{when } k+l+2\leq i\leq k+l+n-1,\\ a_i=0 & \text{when } i\leq k+l-1 \text{ or } i\geq k+l+n+2,\\ a_i\in[0,1] & \text{otherwise}. \end{array} \]
Comparing \eqref{eq:auxonR} with \eqref{eq:auxonZ} it immediately follows that
\[ \big| A_n(F,G)(k+\alpha,l+\beta) - \widetilde{A}_n(\widetilde{F},\widetilde{G})(k,l) \big| \leq \frac{1}{n}\sum_{i\in\{0,1,n,n+1\}} \big|\widetilde{F}(k+i,l) \,\widetilde{G}(k,l+i)\big|, \]
so for any $n\in\mathbb{N}$ we get
\[ \big\| A_n(F,G)(k+\alpha,l+\beta) - \widetilde{A}_n(\widetilde{F},\widetilde{G})(k,l) \big\|_{\ell^2_{(k,l)}(\mathbb{Z}^2)} \leq \frac{4}{n}. \]
Observe that this estimate is uniform in $\alpha,\beta\in[0,1)$. Consequently,
\begin{align*}
\Big| \|A_{n_j}(F,G)(k+\alpha,l+\beta)-A_{n_{j-1}}(F,G)(k+\alpha,l+\beta)\|_{\ell^2_{(k,l)}(\mathbb{Z}^2)} & \\[-1ex]
- \|\widetilde{A}_{n_j}(\widetilde{F},\widetilde{G})-\widetilde{A}_{n_{j-1}}(\widetilde{F},\widetilde{G})\|_{\ell^2(\mathbb{Z}^2)} & \Big| \leq \frac{8}{n_{j-1}},
\end{align*}
so, taking the $\textup{L}^2([0,1)^2)$ norm in $(\alpha,\beta)$,
\[ \Big| \|A_{n_j}(F,G)-A_{n_{j-1}}(F,G)\|_{\textup{L}^2(\mathbb{R}^2)} - \|\widetilde{A}_{n_j}(\widetilde{F},\widetilde{G})-\widetilde{A}_{n_{j-1}}(\widetilde{F},\widetilde{G})\|_{\ell^2(\mathbb{Z}^2)} \Big| \leq \frac{8}{n_{j-1}}. \]
Combining this with \eqref{eq:applytheorem2} and using $\sum_{j=1}^{m} n_{j-1}^{-2} \leq \sum_{n=1}^{\infty} n^{-2} \lesssim 1$ we conclude
\[ \sum_{j=1}^{m} \big\| \widetilde{A}_{n_j}(\widetilde{F},\widetilde{G}) - \widetilde{A}_{n_{j-1}}(\widetilde{F},\widetilde{G}) \big\|_{\ell^{2}(\mathbb{Z}^2)}^{2} \lesssim 1. \]
If we multiply the right hand side by $\|\widetilde{F}\|^2_{\ell^4(\mathbb{Z}^2)} \|\widetilde{G}\|^2_{\ell^4(\mathbb{Z}^2)}$, then by homogeneity the inequality remains to hold for arbitrary $\widetilde{F},\widetilde{G}$ and this establishes Corollary~\ref{cor:newdiscest}.

Finally, we transfer to the measure-preserving system $(X,\mathcal{F},\mu,S,T)$. Let $f,g\in \textup{L}^4(X)$ be normalized by
$\|f\|_{\textup{L}^4(X)}=\|g\|_{\textup{L}^4(X)}=1$.
Take a point $x\in X$ and fix a positive integer $N\geq n_m$. The function $\widetilde{F}_{x,N}\colon\mathbb{Z}^2\to\mathbb{R}$ defined by
\[ \widetilde{F}_{x,N}(k,l) := \begin{cases} f(S^k T^l x) & \text{if } 0\leq k,l\leq 2N-1,\\ 0 & \text{otherwise} \end{cases} \]
and analogously defined $\widetilde{G}_{x,N}$ keep track of the values of $f$ and $g$ along the forward trajectory of $x$. Observe that for integers $0\leq k,l< N$ and $0<n\leq N$ we have
\[ M_{n}(f,g)(S^k T^l x) = \frac{1}{n}\sum_{i=0}^{n-1} f(S^{k+i} T^l x) g(S^k T^{l+i} x) = \widetilde{A}_n\big(\widetilde{F}_{x,N},\widetilde{G}_{x,N}\big)(k,l), \]
where we used $ST=TS$ and the definition \eqref{eq:defineatilde}. The fact that $S$ and $T$ are measure-preserving enables us to write
{\allowdisplaybreaks\begin{align*}
& \|M_{n_j}(f,g)-M_{n_{j-1}}(f,g)\|_{\textup{L}^{2}(X)}^{2} = \int_{X} \big|M_{n_j}(f,g)(x)-M_{n_{j-1}}(f,g)(x)\big|^{2} d\mu(x) \\
& = \frac{1}{N^2} \int_X \sum_{k,l=0}^{N-1} \big|M_{n_j}(f,g)(S^k T^l x) - M_{n_{j-1}}(f,g)(S^k T^l x)\big|^{2} d\mu(x) \\
& \leq \frac{1}{N^2} \int_X \big\| \widetilde{A}_{n_j}(\widetilde{F}_{x,N},\widetilde{G}_{x,N}) - \widetilde{A}_{n_{j-1}}(\widetilde{F}_{x,N},\widetilde{G}_{x,N}) \big\|_{\ell^{2}(\mathbb{Z}^2)}^{2} d\mu(x)
\end{align*}}
for each $1\leq j\leq m$. Similar computation as above gives
\[ 1 = \|f\|_{\textup{L}^{4}(X)}^{4}
= \frac{1}{4N^2} \int_X \sum_{k,l=0}^{2N-1} |f(S^k T^l x)|^4 d\mu(x)
= \frac{1}{4N^2} \int_{X} \|\widetilde{F}_{x,N}\|_{\ell^4(\mathbb{Z}^2)}^{4} d\mu(x). \]
Taking $\widetilde{F}=\widetilde{F}_{x,N}$, $\widetilde{G}=\widetilde{G}_{x,N}$ in Corollary \ref{cor:newdiscest} gives
\[ \sum_{j=1}^{m} \big\| \widetilde{A}_{n_j}(\widetilde{F}_{x,N},\widetilde{G}_{x,N}) - \widetilde{A}_{n_{j-1}}(\widetilde{F}_{x,N},\widetilde{G}_{x,N}) \big\|_{\ell^{2}(\mathbb{Z}^2)}^{2}
\lesssim \|\widetilde{F}_{x,N}\|^4_{\ell^4(\mathbb{Z}^2)} + \|\widetilde{G}_{x,N}\|^4_{\ell^4(\mathbb{Z}^2)}. \]
Integrating this inequality in $x$ over $X$ and dividing by $N^2$ yields
\[ \sum_{j=1}^{m} \|M_{n_j}(f,g)-M_{n_{j-1}}(f,g)\|_{\textup{L}^{2}(X)}^{2} \lesssim 1 \]
for any $n_0<n_1<\cdots<n_m$. This completes the proof of Theorem~\ref{thm:ergodicthm}.


\section{Appendix}
The following inequality \eqref{eq:appineq1} is taken from \cite{jsw:var}; we reproduce a proof for the
convenience of the reader. An alternative inequality serving the same purpose appears in \cite{mst:vr}.

\begin{lemma}\label{lemma:peetre}
If $a\colon[2^i, 2^{i+1}] \rightarrow \mathbb{R}$ is a continuously differentiable function, then
\begin{align}\label{eq:appineq1}
& \sup_{2^i\leq t_0<\cdots<t_m\leq 2^{i+1}}\sum_{j=1}^m |a(t_{j})-a(t_{j-1})|^2 \lesssim \|a(t)\|_{\textup{L}_t^2((2^i,2^{i+1}),dt/t)} \|ta'(t)\|_{\textup{L}_t^2((2^i, 2^{i+1}),dt/t)},\\
\label{eq:appineq2}
& \sup_{2^i\leq t_0<\cdots<t_m\leq 2^{i+1}}\sum_{j=1}^m |a(t_{j})-a(t_{j-1})|^2 \leq \|ta'(t)\|^2_{\textup{L}_t^2((2^i,2^{i+1}),dt/t)}.
\end{align}
\end{lemma}

\begin{proof}
To obtain \eqref{eq:appineq1} we first show that for any $2^i\leq t_0<\cdots<t_m\leq 2^{i+1}$ and each index $1\leq j\leq m$ one has
\begin{equation}\label{eq:appest}
|a(t_{j})-a(t_{j-1})|^2 \lesssim \|a(t)\|_{\textup{L}_t^2((t_{j-1},t_{j}),dt/t)}\|ta'(t)\|_{\textup{L}_t^2((t_{j-1},t_{j}),dt/t)}.
\end{equation}
It suffices to prove this under the assumptions that $a$ is non-negative and absolutely continuous. Indeed, in general we then split $a=a_{+} - a_{-}$, where
$a_{+} = \max(a,0)$ and $a_{-}=-\min(a,0)$. Note that $a_{+}$, $a_{-}$ and satisfy the required properties and that
\[ \|a_{+}(t)\|_{\textup{L}_t^2((t_{j-1},t_{j}),dt/t)} \leq \|a(t)\|_{\textup{L}_t^2((t_{j-1},t_{j}),dt/t)},\; \|ta_{+}'(t)\|_{\textup{L}_t^2((t_{j-1},t_{j}),dt/t)} \leq \|ta'(t)\|_{\textup{L}_t^2((t_{j-1},t_{j}),dt/t)} \]
and analogously for $a_{-},\,a_{-}'$. Using the triangle inequality and applying \eqref{eq:appest} to $a_{+}$ and $a_{-}$ we obtain the inequality for any real-valued absolutely continuous function $a$.

Let us assume that $a$ is as claimed above. Then
\begin{align*}
|a(t_{j})-a(t_{j-1})|^2 \leq \big|a(t_{j})^2-a(t_{j-1})^2\big| \leq \Big| \int_{t_{j-1}}^{t_{j}} t(a(t)^2)' \frac{dt}{t} \Big| =\Big| \int_{t_{j-1}}^{t_{j}} 2a(t)ta'(t) \frac{dt}{t}\Big|.
\end{align*}
Applying the Cauchy-Schwarz inequality in $t$ we bound this up to a constant by
\begin{align*}
\Big( \int_{t_{j-1}}^{t_{j}} a(t)^2 \frac{dt}{t} \Big)^{1/2}\Big( \int_{t_{j-1}}^{t_{j}} (ta'(t))^2 \frac{dt}{t} \Big)^{1/2},
\end{align*}
which shows \eqref{eq:appest}. Summing over $j$ and applying the Cauchy-Schwarz inequality we obtain
\begin{align*}
\sum_{j=1}^m |a(t_{j})-a(t_{j-1})|^2 & \lesssim \Big(\sum_{j=1}^m \|a(t)\|^2_{\textup{L}_t^2((t_{j-1},t_{j}),dt/t)}\Big)^{1/2} \Big(\sum_{j=1}^m \|ta'(t)\|^2_{\textup{L}_t^2((t_{j-1},t_{j}),dt/t)}\Big)^{1/2} \\
& \leq \|a(t)\|_{\textup{L}^2_t((2^i, 2^{i+1}),dt/t)} \,\|ta'(t)\|_{\textup{L}^2_t((2^i, 2^{i+1}),dt/t)}
\end{align*}
for any $2^i\leq t_0<\cdots<t_m\leq 2^{i+1}$, which establishes \eqref{eq:appineq1}.

To see \eqref{eq:appineq2} we estimate
\begin{align*}
|a(t_{j})-a(t_{j-1})|^2 = \Big| \int_{t_{j-1}}^{t_j} t a'(t)\frac{dt}{t} \Big|^2 \leq (t_j-{t_{j-1}}) \int_{t_{j-1}}^{t_{j}} (ta'(t))^2 \frac{dt}{t^2} \leq 2^i\int_{t_{j-1}}^{t_{j}} (ta'(t))^2 \frac{dt}{t^2}.
\end{align*}
The first inequality follows from the Cauchy-Schwarz inequality in $t$, while for the second inequality we used the crude bound $t_j-t_{j-1}\leq 2^i$. Thus,
\begin{align*}
\sum_{j=1}^m |a(t_{j})-a(t_{j-1})|^2 \leq 2^i\int_{2^i}^{2^{i+1}} (ta'(t))^2 \frac{dt}{t^2} \leq \int_{2^i}^{2^{i+1}} (ta'(t))^2 \frac{dt}{t},
\end{align*}
which gives \eqref{eq:appineq2}.
\end{proof}


\section*{Acknowledgments}
P. D. and C. T. are supported by the Hausdorff Center for Mathematics.
V. K. and K. A. \v{S}. are supported in part by the Croatian Science Foundation under the project 3526.
All four authors are also supported by the bilateral DAAD-MZO grant \emph{Multilinear singular integrals and applications}. We thank the anonymous referee for pointing out that our prior result for short variation in norm extends to a result for the pointwise short variation.


\end{document}